\documentclass[11pt]{amsart}
	\usepackage{amssymb, latexsym, amsmath, amsfonts}
	\usepackage{bm}
	\usepackage{wrapfig,upgreek}
	\usepackage{mathrsfs}
	\usepackage[pdftex]{graphicx}
	\usepackage{float}
	\usepackage[numbers]{natbib}
	\usepackage{tikz}
	\usepackage{hyperref}
	\usepackage{changepage}
	\newtheorem{thm}{Theorem}[section]
	\newtheorem{cor}[thm]{Corollary}
	\newtheorem{lem}[thm]{Lemma}
	\newtheorem{prop}[thm]{Proposition}
	\theoremstyle{definition}

	\theoremstyle{remark}
	\newtheorem{rem}[thm]{Remark}
	\numberwithin{equation}{section}
	
	\numberwithin{equation}{section}

\vbadness=\maxdimen
	\newcommand{\mbb}{\mathbb}

	\newcommand{\ov}{\overline}
	
	\newcommand{\ep}{\epsilon}
	\newcommand{\no}{\noindent}
	
	\newcommand{\Om}{\Omega}
	\newcommand{\De}{\Delta}
	\newcommand{\cal}{\mathcal}

	\newcommand{\norm}[1]{\|#1\|}
	\newcommand{\short}[1]{\textit{Short}  $\mathbb C^{#1}$}
	
	\newcommand{\om}[1]{\Omega_{\{#1_n\}}}
	\newcommand{\J}[1]{J^+_{\{#1_n\}}}
	\newcommand{\K}[1]{K^+_{\{#1_n\}}}
	\newcommand{\seq}[1]{\{#1_n\}}
	
\begin{document}
	\title{Examples of non--autonomous basins of attraction--II}
	
	\author{Sayani Bera} 
	\address{Sayani Bera: School of Mathematics, Ramakrishna Mission Vivekananda Educational and Research Institute, 711202, West Bengal, India}
	\email{sayanibera2016@gmail.com}
	
	\pagestyle{plain}
	\begin{abstract}
	The aim of this article is to enlarge the list of examples of non--autonomous basins of attraction from \cite{ShortCk} and at the same time explore some other properties that they satisfy. For instance, we show the existence of countably many disjoint \short{k}'s in $\mbb C^k.$ We also construct a \short{k} which is not Runge and exhibit yet another example whose boundary has Hausdorff dimension $2k.$
	\end{abstract}
	\maketitle
	\section{Introduction}
\no We continue our work on non--autonomous basins of attraction, in particular on \short{k}'s from \cite{ShortCk}. Recall that a \short{k} is a proper subdomain of  $\Omega \subset \mbb C^k$ satisfying the following properties:
\begin{itemize}
\item[(i)] $\Omega =\bigcup \Omega_n$, where each $\Om_n$ is biholomorphic to the ball $B^k(0;1)$,
\item[(ii)] The infinitesimal Kobayashi metric vanishes identically, i.e., $k_{\Omega} \equiv 0$.
\item[(iii)] $\Omega$ admits a non--constant plurisubharmonic function that is bounded above.
\end{itemize}
These were first constructed by Forn\ae ss \cite{ShortC2} who showed that they can be obtained as non--autonomous basins of attraction of a sequence of automorphisms $\seq{F} \in \mathsf{ Aut}_0(\mbb C^k)$ of the form
\begin{align}\label{short k} 
 F_n(z_1,z_2,\hdots,z_k)=(z_1^d+ a_n z_k,a_n z_1,\hdots,a_n z_{k-1})
 \end{align}
where $0<|a_{n+1}|< |a_n|^d <1$ for every $n \ge 0.$ 

\medskip\no Recall that a sequence of automorphisms $\seq{F}\in \mathsf{ Aut}_0(\mbb C^k)$ is said to be \textit{uniformly bounded} at the origin if there exists real constants $0< C<D<1$ and $r>0$ such that 
\[
C\|z\| \le  \|F_n(z)\| \le D\|z\|
\]
for every $z \in B^k(0;r)$ and $n \ge 0.$ 

\medskip\no Note that the maps $F_n$ in (\ref{short k}) are not uniformly bounded at the origin. The purpose of \cite{ShortCk} was to give examples of \short{k}'s that were motivated by the existing examples of Fatou--Bieberbach domains. Here we will extend this list by providing more examples of \short{k}'s -- in fact these examples will be  biholomorphic images of non--autonomous basins of sequences of automorphisms satisfying the {\it uniform upper--bound} condition. Here, a sequence $\seq{F} \subset \mathsf{ Aut}_0(\mbb C^k)$ is said to satisfy the \textit{uniform upper--bound} condition at the origin if there exists $r>0$ and $0<C<1$ such that
$$\|F_n(z)\|< C\|z\|$$
for every $z \in B^k(0;r).$ Henceforth, the basin of attraction at the origin of such a sequence will always be denoted by $\om{F}.$ The results are organized in the following sequence:

\medskip \no  
In Section \textbf{2}, we show that there exist countably many disjoint \short{k}'s in $\mbb C^k$, $k \ge 2$ and there exist \short{k}'s which are not Runge whenever $k \ge 4.$ These results follow directly as an application of the fact that Fatou--Bieberbach domains with the aforementioned properties are known to exist from Rosay--Rudin \cite{RR1} and Wold \cite{Runge} respectively.

\medskip\no 
In Section \textbf{3}, we give two alternative methods to construct biholomorphic images of non--autonomous basins of attraction.

\medskip\no The first result here is specific to \short{2}'s. However it can be proved for $k \ge 2$, from Remark \ref{shift}. Now \short{k}'s are constructed as the non--autonomous basin of attraction of sequence $\seq{F} \in \mathsf{ Aut}_0(\mbb C^k)$ where $F_n$'s are as in (\ref{short k}). Our aim is to show that the sequence $\seq{F}$ can involve higher order terms, i.e.,
\begin{align*}
 F_n(z_1,z_2,\hdots,z_k)=(z_1^d+o(z_1^{d+1})+ a_n z_k,a_n z_1,\hdots,a_n z_{k-1})
 \end{align*}
where $0<|a_{n+1}|< |a_n|^d <1$ for every $n \ge 0$ and still the basin of attraction at the origin (i.e., $\om{F}$) is a \short{k}. Theorem \textbf{\ref{poly_short}} is achieved as an effort towards proving the same which is stated as follows:
\begin{thm}\label{poly_short}
Let $\{a_n\}$ be a strictly positive sequence in $\mbb R$ such that $\lim_{n \to \infty} a_n^{-2^n}=0$ and $0<a_{n+1}<a_n^2 <1.$ If $\{F_n\} \subset \mathsf{ Aut}_0(\mbb C^2)$ of the form
\[F_n(z_1,z_2)=(a_n z_2+z_1^2P(z_1),a_n z_1)\] where $P$ is a polynomial in one variable with positive coefficients and $P(0)=c_0 > 0$ then the basin of attraction at the origin (i.e., $\Omega_{\{F_n\}}$) is a \short{2}.
\end{thm}
\no However, it only says that $z_1^d$ can be replaced with a polynomial in $z_1$ provided there is a restriction on the coefficients of the polynomial and the order of convergence of $|a_n|$'s.

\medskip\no Next, we show that a non--autonomous basins of attraction that satisfies the \textit{uniform upper--bound} condition at a point is a parabolic domain realized as an increasing union of domains biholomorphic to the ball. Also, we give a sufficient condition for the existence of biholomorphisms between two such non--autonomous basin, i.e., given a sequence of automorphisms $\seq{F} \in \mathsf{ Aut}_0(\mbb C^k)$ satisfying the \textit{uniform upper--bound} condition at the origin, each function can be sufficiently perturbed in a small enough ball at the origin so that the basin of attraction of the resulting sequence is biholomorphic to $\om{F}.$ This result is motivated from push--out methods due to  Dixon--Esterle \cite{DE}, Glovebnik \cite{Gl2} and Stens\o nes \cite{St}, \cite{GlSt}.
\begin{thm}\label{transcendence}
Let $\{S_n\} \subset \mathsf{ Aut}_0(\mbb C^k)$, $k \ge 2$ satisfy the \textit{uniform upper--bound condition} and $\seq{F} \subset \mathsf{ Aut}_0(\mbb C^k)$. Then there exists a sequence of positive real numbers $\seq{\delta}$, $\delta_n \to 0$ as $n \to \infty$ and $r>0$ such that $\om{F} \cong \om{S}$ if
\[ \|F_n(z)-S_n(z)\|< \delta_n\] for every $z \in B^k(0;r).$
\end{thm}
\no In Section \textbf{4}, Theorem \textbf{\ref{poly_short}} is applied to give a constructive proof of the existence of \short{k}'s with chaotic boundary, i.e., there exists a \short{k} in $\mbb C^k$ such that the upper box--dimension of the boundary is strictly greater than $2k-1$ However, an existential proof of a much more stronger result will be achieved as an application of Theorem \textbf{\ref{transcendence}}.

\medskip\no 
In Sections \textbf{5} and \textbf{6}, we apply Theorem \textbf{\ref{transcendence}} to obtain results about biholomorphic images of non--autonomous basins of attraction at a point satisfying the \textit{uniform upper--bound }condition. The analogs of these results for Fatou--Bieberbach domains are known to be true from \cite{PW}, \cite{Wold} and \cite{Runge}. Our methods are adopted from the techniques in these articles. The new ingredient that we used is Theorem \textbf{\ref{transcendence}}. Here are our results:
\begin{thm}\label{poly_con_1}
Let $K$ be a polynomially convex compact subset of $\mbb C^k$ and let $\{p_j\} \subset \mbb C^k \setminus K$ and $\seq{S} \in \mathsf{ Aut_0}(\mbb C^k)$ be a sequence that satisfies the uniform upper--bound condition at the origin. Then there exists a biholomorphism $\Phi: \om{S} \to \mbb C^k$, such that $\{p_j\} \subset \Phi(\om{S}) \subset \mbb C^k \setminus K$.
\end{thm}
\begin{cor}\label{compact}
Given a polynomially convex compact set $K$ and a sequence of automorphisms  $\seq{S} \in \mathsf{ Aut_0}(\mbb C^k)$ that satisfies the uniform upper--bound condition at the origin, there exists a biholomorphism $\Phi: \om{S} \to \mbb C^k$, such that $\Phi(\om{S})$ is dense in $\mbb C^k \setminus K.$
\end{cor}
\begin{cor}\label{Runge}
Given a sequence of automorphisms  $\seq{S} \in \mathsf{ Aut_0}(\mbb C^k)$ that satisfies the uniform upper--bound condition at the origin, there exists a biholomorphism $\Phi:\om{S} \to \mbb C^k$  such that $\Phi(\om{S})$ is not Runge.
\end{cor}
\begin{thm}\label{dense}
Given any sequence of automorphisms $\seq{S} \in \mathsf{ Aut}_0(\mbb C^k)$ that satisfy the uniform upper--bound condition at the origin and any $m \in \mbb N \cup \{\infty\}$ there exist $m-$biholomorphisms $\{\Phi_i: 1 \le i \le m\}$ such that the following hold:
\begin{itemize}
\item[(i)] $\Phi_i(\om{S}) \cap \Phi_j(\om{S})=\emptyset$ whenever $1 \le i \neq j \le m.$

\item[(ii)] Let $\Om=\cup \Phi_i(\om{S})$. For any $q \in \mbb C^k\setminus \Om$, $q \in \partial\Phi_i(\om{S})$ for every $1 \le i \le m.$
\end{itemize} 
\end{thm}
\begin{thm}\label{hausdorff}
Given any sequence of automorphisms $\seq{S} \in \mathsf{ Aut}_0(\mbb C^k)$ that satisfy the uniform upper--bound condition at the origin, there exists a biholomorphism $\Phi: \om{S} \to \mbb C^k$ such that the Hausdorff dimension at any point in the boundary of $\Phi(\om{S})$ is $2k.$
\end{thm}

\medskip\no 
\no {\it Acknowledgements:} The author would like to thank Han Peters and Kaushal Verma for suggesting the problems. The author would also like to thank Luka Boc--Thaler for suggesting Proposition \textbf{\ref{luka}}.  
\section{ Examples of Short \texorpdfstring {$\mathbb{C}^k$}{}'s}
\no First we prove that there exist countably many disjoint \short{k}'s in $\mbb C^k.$ Recall that a sequence of points $\{p_j\}$ in $\mbb C^k$ is said to be {\it tame}, if there exists an automorphism $\phi \in {\sf Aut}(\mbb C^k)$ such that 
\[ \phi(p_j)=je_1\]
where $e_1=(1,0, \hdots,0).$
\begin{prop}\label{infinite}
Given a tame sequence of point in $\mbb C^k$, there exists a collection of disjoint \textit{Short} $\mbb C^k$'s centered at each point of the tame sequence.
\end{prop}

\begin{proof}
As noted in Rosay--Rudin \cite{RR1}, the automorphism $F$ of $\mbb C^2$ given by
\[F(z_1,z_2)=\big(z_1+z_2,\frac{1}{2}(1-z_2-e^{z_1+z_2})\big)\]  has an attracting fixed point at each $p_m=(2m\pi i,0)$ for every $m \ge 0$. Now, given a tame sequence, say $\{a_m\}$ in $\mbb C^k$, $k \ge 2$, there exists an automorphism $f_1$ of $\mbb C^k$ such that 
\[ f_1(a_m)=2 \pi i m e_1\] where $e_1=(1,0,\hdots,0).$ Let
\[ f_2(z_1,z_2,\hdots,z_k)=\big(F(z_1,z_2),az_3,\hdots,az_k\big)\] for $0<|a|<1.$ This is an automorphism of $\mbb C^k.$ Clearly, $2 \pi i m e_1$ is an attracting fixed point $f_2$ for each $m \ge 0$ and the corresponding attracting basin of $f_2$ for each $2 \pi i m e_1$ (say $\Omega_m$) is a Fatou--Bieberbach domain, i.e., there exist biholomorphisms $\psi_m: \Omega_m \to \mbb C^k$ for every $m \ge 0.$ Also, without loss of generality one can assume that $\psi_m(2 \pi i m e_1)=0.$ 

\medskip\no 
Now from \cite{ShortC2}, there exists a \textit{Short} $\mbb C^k$, say $\omega$ obtained as a non--autonomous of basin of attraction at the origin. Let $\omega_m=\psi_m^{-1}(\omega).$ Thus, $\omega_m$ is a \textit{short} $\mbb C^k$. Let $\Psi_m=f_1^{-1}\circ \psi_m^{-1} $ and $\widetilde{\omega}_m=\Psi_m(\omega).$ Then $\widetilde{\omega}_m$ is the required disjoint collection of \textit{Short} $\mbb C^k$'s.
\end{proof}
\no Let ${\sf Aut}_0(\mbb C^k)$ denote the group of automorphisms of $\mbb C^k$ that fixes the origin and for a sequence $\{F_n\} \in {\sf Aut}_0(\mbb C^k)$ let 
\[ F(n)(z)=F_n \circ \cdots \circ F_1(z).\] 
\begin{prop}\label{Omega times Cl}
Let $\{F_n\} \in \mathsf{ Aut}_0(\mbb C^k)$, $k \ge 2$ be a sequence of automorphisms  such that the basin of attraction at $\om{F}$ is a \short{k}. Then for every $l \ge 1$, $\mbb C^l \times \om{F}$ is a \textit{Short} $\mbb C^{l+k}.$
\end{prop}
\begin{proof}
Since $\om{F}$ is a \textit{Short} $\mbb C^k$ it satisfies the following properties:
\begin{enumerate}
\item[(i)] $\om{F}$ is a non--empty open connected set of $\mbb C^k$.

\medskip
	\item[(ii)] $\om{F}=\cup_{j=1}^{\infty}\Omega_j$, $\Omega_j \subset \Omega_{j+1}$, and each $\Omega_j$ is biholomorphic to the unit ball $B^k(0;1)$ in $\mbb C^k.$ Further, for a given $0< c< 1$ there exists $n_0 \ge 1$ such that
	\[ \Omega_j= F(n_0+j)^{-1}(B^k(0;c))\] 
	\item[(iii)] The infinitesimal Kobayashi metric on $\om{F}$ vanishes identically.
	
	\medskip
	\item[(iv)] There exists a non--constant plurisubharmonic function $\phi : \om{F} \to [-\infty , \infty)$ such that
	\[\om{F}=\{z \in \mbb C^k: \phi(z)<0\}.\]
	\end{enumerate}
Clearly $\mbb C^l \times \om{F}$ is an open connected set in $\mbb C^{l+k}.$
For each $n \ge 1$, let
$$\tilde{F}_n(z_1,z_2,\hdots,z_{k+l})=(\alpha z_1,\alpha z_2,\hdots,\alpha z_l, F_n(z_{l+1},\hdots,z_{k+l}))$$ where $0<|\alpha|<1$ and
\[ \widetilde{\Omega}_j=\tilde{F}(n_0+j)^{-1}\big(B^l(0;c) \times B^k(0;c)\big)=B^l(0;\alpha^{-(n_0+j)}c) \times \Omega_j.\]
Note that $\mbb C^l \times \om{F}$ is the basin of attraction of the sequence $\{\tilde{F}_n\}$ at the origin and $$\mbb C^l \times \om{F}= \bigcup_{j \ge 0} \widetilde{\Omega}_j.$$ Let $U_j=\tilde{F}(j)^{-1}\big(B^{l+k}(0;c)).$ Then clearly, $U_{n_0+j} \subset \widetilde{\Omega}_j.$ Since $\Om_j \subset \om{F}$, there exists $l_0 \ge 1$ such that for every $z \in \widetilde{\Omega}_j$
\[ \tilde{F}\big(n_0+j+l\big)(z) \in B^l(0; \alpha^{l_0} c) \times B^k(0;(c')^{l_0} c) \subset  B^{l+k}(0;c)\] where $0<c'<1.$
Hence, $\widetilde{\Omega}_j \subset U_{n_0+j+l_0}.$ Also for sufficiently large $n$, $U_n \subset U_{n+1}$ and thus $\mbb C^l \times \om{F}= \bigcup_{j \ge 0}U_{n_0+j+l_0}.$ 

\medskip\no 
Let $p \in \mbb C^l \times \om{F}$ and $\xi \in T_p(\mbb C^l \times \om{F}).$ Let $p'=(p_1,\hdots,p_l)$, $p''=((p_{l+1},\hdots,p_{l+k})$, $\xi'=(\xi_1,\hdots,\xi_l))$ and $\xi''=(\xi_{l+1},\hdots,\xi_{l+k}).$ Since $\tilde{F}_n$ is a linear map in the first $l-$variables and $\om{F}$ is a \textit{Short} $\mbb C^k$, there exist $F_1: \Delta(0;1) \to \mbb C^l$ such that $$F_1(0)=p'\text{ and }F_1'(0)=R \xi'$$ and $F_2: \Delta(0;1) \to \om{F}$ such that $$F_2(0)=p''\text{ and }F_2'(0)=R \xi''$$ for every $R>0.$ Let $F(z)=(F_1(z),F_2(z)).$ Thus $F(0)=p$ and $F'(0)=R \xi.$ But this is true for any $R>0$, and hence the infinitesimal Kobayashi metric vanishes on $\mbb C^l \times \om{F}.$

\medskip\no 
Let $\tilde{\phi}(z_1,\hdots,z_{k+l})=\phi(z_{l+1},\hdots,z_{k+l}).$ Since $\tilde{\phi}$ is independent of the first $l-$variables, $\tilde{\phi}$ is plurisubharmonic on $\mbb C^l \times \om{F}$ and
$$\mbb C^l \times \om{F}=\{z \in \mbb C^{l+k}: \tilde {\phi}(z)<0\}.$$ Thus $\mbb C^l \times \om{F}$ is a non--autonomous of basin of attraction and is a \textit{Short} $\mbb C^{l+k}.$
\end{proof}
\begin{cor}\label{Fatou-short Lemma}
Let $\Omega_1 \subset \mbb C^l,$ $ l \ge 2$ be a Fatou--Bieberbach domain and $\Omega_2 \subset \mbb C^k$, $k \ge 2$  be a \textit{Short} $\mbb C^k$. Then $\Omega_1 \times \Omega_2$ is a \textit{Short} $\mbb C^{l+k}$.
\end{cor}
\begin{proof}
Let $\phi: \Omega_1 \to \mbb C^l$ be a biholomorphism that identifies $\Omega_1$ with $\mbb C^l$. Then $\tilde{\phi}: \Om_1 \times \Om_2 \to \mbb C^l \times \Om_2$ as $\tilde{\phi}(z_1,z_2)=(\phi(z_1),z_2)$, for $z_1 \in \Om_1$ and $z_2 \in \Om_2$ is evidently a biholomorphism. Hence $\Omega_1 \times \Omega_2$ is a \short{l+k}.
\end{proof}
\begin{cor}
For each $k \ge 4$ there exists a \short{k} which is not Runge.
\end{cor}
\begin{proof}
By Theorem \textbf{1} in \cite{Runge}, there exists a Fatou--Bieberbach domain (say $\Omega_1$) in $\mbb C^2$ which is not Runge. From Lemma \textbf{\ref{Fatou-short Lemma}}, if $\Omega_2$ is a \short{k-2}, then $\Om_1 \times \Om_2$ is \short{k}. 

\medskip\no \textit{Claim: }$\Om_1 \times \Om_2$ is not Runge.

\medskip\no Since $\Om_1$ is not Runge there exists a compact set 
$K$ such that the polynomial convex hull of $K$, $\widehat{K} \not \subset \Om_1.$ Now fix a $w_0 \in \Om_2$ and define the following sets:
\[ K_{w_0} =\{(z,w_0) \in \mbb C^{l+k}: z \in K\} \text{ and } \widehat{K}_{w_0}=\{(z,w_0) \in \mbb C^{l+k}: z \in \hat{K}\}.\]
As $\widehat{K}\not \subset \Om_1$, $\widehat{K}_{w_0} \not \subset \Om_1 \times \Om_2. $ Suppose $P$ be  a polynomial map from $\mbb C^{l+k}$ and $(z,w_0) \in \widehat{K}_{w_0}.$ Then $P_{w_0}(z)=P(z,w_0)$ is polynomial in $\mbb C^l$ and $|P_{w_0}(z)| \le \|P_{w_0}\|_K$, i.e., $|P(z,w_0)| \le \|P\|_{K_{w_0}}.$ Hence $\widehat{K}_{w_0} \subset \widehat{K_{w_0}} \not \subset \mbb C^{l+k}.$ Thus the proof.
\end{proof}

\no Here is an alternative proof of the existence of a non-Runge \short{k}, $k \ge 2$, that was suggested to us by Luka Boc--Thaler. Recall the following fact from \cite{Runge}.
\begin{thm}\label{Runge_pre}
There exists a subset $Y \subset \mbb C^* \times \mbb C$, such that $0 \in \widehat{Y}.$ Further, for any $p \in \mbb C^* \times \mbb C$ and $\ep>0$, there exists a biholomorphism of $\psi \in \mathsf{ Aut}(\mbb C^* \times \mbb C)$ such that $\psi(Y)=B^2(p;\ep).$
\end{thm} 
\begin{prop}\label{luka}
 For every $k \ge 2$, there exists a \short{k} which is not Runge.
\end{prop}
\begin{proof}
From \cite{Runge} there exists a Fatou--Bieberbach domain, $D$ which is contained in $\mbb C^* \times \mbb C^{k-1}.$ Let $\phi: D \to \mbb C^k$, be the biholomorphism, then $\omega=\phi(\Omega)$ is a \short{k} where $\Omega$ is a \short{k}. Now from Theorem \textbf{\ref{Runge_pre}}, there exist an automorphism, $\psi$ of $\mbb C^* \times \mbb C^{k-1}$ such that for $B^k(p;\ep) \subset \omega$, $\psi(B^k(p;\ep))=Y.$ Hence $\psi(\omega)$ is a non--Runge \short{k}.
\end{proof}
\section{Proof of Theorems \ref{poly_short} and \ref{transcendence}}
\begin{proof}[Proof of Theorem \ref{poly_short}]
Note that there exists $M>0$ such that 
\[ |z_1^2P(z_1)|< M|z_1|^2 \text{ for every } z_1 \in D(0;r)\] whenever $0<r<1.$ Let $0<c<1$ such that $(1-Mc)>0.$ Further, choose $0<c'<c$ and let $c_l=c(c')^l.$

\medskip\no 
\textit{Claim:} For every $n \ge n_0$, sufficiently large and $l \ge 0$ then $F_{n+l}(\Delta^2(0;c_l)) \subset \Delta^2(0;c_{l+1}).$

\medskip\no 
Since $\lim_{n \to \infty} a_n^{-2^n}=0$ there exists $0< a<1$ such that for $n$ sufficiently large and every $l \ge 0$
\begin{align*}
\log a_{n+l} &= 2^{n+l} \log a_{n+l}^{-2^{n+l}} \\
& \le (l+1)2^n \log a.  
\end{align*}
Thus, for some $n \ge n_0$ and every $l \ge 0$
\begin{align*}
\log a_{n+l}&< \log c(1-Mc)+(l+1) \log c' \\
a_{n+l} &< c(1-Mc)c'^{l+1} < c(c')^{l+1}-M(c(c')^l)^2 <c_{l+1}-Mc_l^2\\
\text{ i.e., } a_{n+l}c_l+Mc_l^2 &< c_{l+1}.
\end{align*}
Hence the claim.

\medskip\no 
	Now define
	\[\Omega_n=\{z \in \mbb C^2: F(n)(z) \in \Delta^2(0;c)\}.\]
From the above claim, $\Omega_{n} \subset \Omega_{n+1}$ for sufficiently large $n$. Also for every $n \ge n_0$ and $l\geq 1$, $F_{n+l} \circ \cdots\circ F_{n+1}(z) \to 0$ uniformly on $\Omega_n$ thus $\cup_{n \ge n_0}\Omega_n \subset \Omega_{\{F_{n}\}}.$ Conversely, if $z \in \Omega_{\{F_n\}}$, then $\|F(n)(z)\|<c$ for sufficiently large $n$, i.e., $z \in \Omega_n$ for $n$ large. Hence $\cup_{n \ge n_0}\Omega_n = \Omega_{\{F_{n}\}}.$ 
\begin{lem}
$\Omega_{\{F_n\}}=\cup_{j \ge 0} U_j$ where $U_j \subset U_{j+1}$ and each $U_j$ is biholomorphic to the unit ball in $\mbb C^2$. Further, the infinitesimal Kobayashi metric vanishes identically on $\Omega_{\{F_n\}}.$
\end{lem}
\begin{proof}
The proof is similar to the proof of (ii) and (iii) of Theorem 1.4 from \cite{ShortC2}.
\end{proof}
%
%

\medskip\no 
Let $F(n)(z)=(f_1^n(z),f_2^n(z))$. Define 
	\[ \phi_n(z)=\max\{|f_1^n(z)|, |f_2^n(z)|,a_n\}.\]
\begin{lem}
Let $$ \psi_n =\frac{1}{2^n}{\log \phi_n}.$$ Then $\psi_n \to \psi$ on $\Omega_{\{F_n\}}$ and $\psi$ is a plurisubharmonic function on $\Omega_{\{F_n\}}.$ 
\end{lem}
\begin{proof}
\medskip\no 
Since $z \in \Omega_{\{F_n\}}$, there exists $n\ge n_z$ such that $\phi_n(z) \le c.$ Since $a_{n+1}\le a_n^2$,
	\begin{align*}
	\phi_{n+1}(z) \le \max\{M\phi_n(z)^2+ a_{n+1}, a_{n+1}\} \le (M+1) \phi_n(z)^2.
	\end{align*}
		Thus  for every $z \in  \Omega_{\{F_n\}}$ 
	\[ \frac{1}{2^{n+1}}\log \phi^{n+1}(z) \le \frac{1}{2^{n+1}}\log M + \frac{1}{2^n}\log \phi_n(z).\] Now define 
	\[ \Phi_n(z)=\frac{1}{2^n}\log \phi_n(z)+\sum_{j \ge n}\frac{1}{2^{j+1}}\log M.\] Thus $\Phi_n$ is a monotonically decreasing sequence of plurisubharmonic functions and hence its limit, i.e., $\psi$ will be plurisubharmonic.
\end{proof}	
	\begin{lem}
	For every $z \in \Omega_{\{F_n\}}$, $\psi(z)<0$ i.e., $\psi$ is a bounded plurisubharmonic function on $\Omega_{\{F_n\}}.$ Further, $\Omega_{\{F_n\}}$ is not all of $\mbb C^2.$
	\end{lem}
	\begin{proof}
The proof is similar to Lemma 3.3 and Lemma 3.4 in \cite{ShortCk}.
	\end{proof}
\begin{lem}
$\psi$ is non--constant on $\Omega_{\{F_n\}}.$
\end{lem} 
\begin{proof}
\no Since $\{F_n\} \subset \mathsf{ Aut}_0(\mbb C^2)$, $\psi(0)=-\infty.$  If $\psi$ is constant, then $\psi\equiv -\infty$ on $\Omega_{\{F_n\}}.$ 

\medskip\no 
\textit{Induction statement:} For $x>0$ and $y>0$, $f_i^n(x,y)>0$ for every $i=1,2$ and $$\phi_n(x,y)> f_1^n(x,y)> c_0^{-1}(c_0 x)^{2^{n+1}}$$ for every $n \ge 0.$ 

\medskip\no 
\textit{Initial step: }It is true since, $\pi_1 \circ F_0(x,y)=a_0 y+x^2P(x)>c_0 x^2>0$ and $\pi_2 \circ F_0(x,y)=a_0 x>0.$

\medskip\no 
\textit{General step: }Assume that $f_1^n(x,y)> c_0^{-1}(c_0 x)^{2^{n+1}}>0$ and $f_2^n(x,y)>0.$ Let 
\[ f_1^n(x,y)= c_0^{-1}(c_0 x)^{2^{n+1}}+ c_n\] for some $c_n>0.$ Now 
$$ f_1^{n+1}(x,y)=a_{n+1}f_2^n(x,y)+(c_0^{-1}(c_0 x)^{2^{n+1}}+ c_n)^2 P(f_1^n(x,y))>c_0^{-1}(c_0 x)^{2^{n+2}}>0$$ and 
$$f_2^{n+1}(x,y)=a_{n+1}f_1^n(x,y)>0.$$
Since $\De^2(0;c) \subset \om{F}$ for $x,y>0$ and $(x,y) \in \De^2(0;c)$ it follows that 
\[ \psi_n(x,y)=\frac{\log c_0^{-1}}{2^{n}}+ \log c_0x \to \log c_0x \neq -\infty \] as $n \to \infty.$ Hence $\psi$ is non--constant on $\om{F}.$
\end{proof}
\no This completes the proof of Theorem \ref{poly_short}.
\end{proof}
\begin{rem}\label{shift}
Let 
$$S_n(z_1,z_2,\hdots,z_k)=(z_1^2P(z_1)+a_n z_k,a_n z_1,\hdots,a_n z_{k-1})$$
be a sequence of shift--like maps in $\mbb C^k$, $k \ge 3$, where $P$ and $\{a_n\}$ are as in Theorem \textbf{\ref{poly_short}}. The same techniques can be adapted to prove that $\Omega_{\{S_n\}}$ (the basin of attraction of $S_n$'s at the origin) is a \short{k}.
\end{rem}
\no Next, we prove a few properties of a non--autonomous basin of attraction, satisfying the \textit{uniform upper--bound} condition. 
\begin{prop}\label{uniform bound}
Let $\seq{S} \in \mathsf{ Aut}_0(\mbb C^k)$, $k \ge 2$ with \textit{uniform upper--bound condition} at the origin. Then $\om{S}$ satisfies the following properties:
\begin{itemize}
\item[(i)] $\om{S}$ is a connected open set in $\mbb C^k.$

\item[(ii)]There exists $r_0>0$ such that for every $0<r \le r_0$, $\Om^S_n \subset \Om^S_{n+1} $ and $$\om{S}=\bigcup_{n \ge 0} \Om^S_n$$ where $\Om^S_n=S(n)^{-1}(B^k(0;r)).$ 
\item[(iii)] The infinitesimal Kobayashi metric vanishes identically on $\om{S}.$
\end{itemize}  
\end{prop}
\begin{proof}
By assumption there exist $r_0>0$ and $C<1$ such that
\[ \|S_n(z)\| \le C\|z\|\] for every $z \in \ov{B^k(0;r_0)}$ and $n \ge 0.$ Further, for every $0< r \le r_0$,  $B^k(0;r) \subset S_n^{-1}(B^k(0;r)).$ Hence $\Om_n^S \subset \Om_{n+1}^S.$ Similar arguments as in the proof of Theorem \textbf{\ref{poly_short}} gives $\om{S}=\cup_{n \ge 0} \Om^S_n.$ This proves (i) and (ii).

\medskip\no 
Fix $p \in \om{S}$ and $\xi \in T_p\om{S}$, then for $p_n \to 0$ and $\xi_n \to 0$ as $n \to \infty$, where
$p_n=F(n)(p) \text{ and } \xi_n=DF(n)\xi.$ For any $R>0$, consider the maps from unit disc, $\eta_n : \Delta(0;1) \to \mbb C^k$ defined as $\eta_n(x)=p_n+xR \xi_n.$ Let $\tau_n=F(n)^{-1} \circ \eta_n.$ Since $\eta_n(\Delta(0;1)) \subset B^k(0;r)$ for $n$ sufficiently large, $\tau_n(\Delta(0;1)) \subset \om{S}.$ Now $\tau_n(0)=p$ and $\tau_n'(0)=R \xi.$ As $R>0$ is arbitrary, (iii) is true.
\end{proof}
\no Finally, we prove Theorem \textbf{\ref{transcendence}}.
\begin{proof}[Proof of Theorem \ref{transcendence}]
By assumption there exist $0<r<1$ and $C<1$ such that
\[ \|S_n(z)\|< C\|z\|\] for every $z \in \ov{B^k(0;r)}$ and $n \ge 0$, i.e.,  $S_n(\ov{B^k(0;r)}) \subset {B^k(0;Cr)}$ for all $n \ge 0$. Let $r_0=Cr$ and $0<\ep< r-Cr<1.$ Then 
\begin{align}\label{inverse 1}
 B^k(0; r_0+\ep) \subset B^k(0;r) \subset S_n^{-1}(B^k(0;r_0))
\end{align}
for every $n \ge 0.$ Let $$M_n=\max\{\|D(S(n-1))^{-1}(z)\|_{\text{op}}: z \in B^k(0;r)\}$$ for every $n \ge 0$ and choose $0<\ep_n<{\ep^{n+1}}/{M_n}.$ Further, let $0<\delta< \min\{\ep,1-C\}$ and $\tilde{C}=C+\delta<1.$ Thus
\begin{align}\label{inverse 1a}
F_n(B^k(0;r)) \subset B^k(r_0+\delta_n) \subset B^k(0;r).
\end{align} 
By continuity of the functions $S_n^{-1}$'s there exists $\tilde{\delta}_n>0$ such that 
\[ \|S_n^{-1}(z)-S_n^{-1}(w)\|< \ep_n\] whenever $\|z-w\|< \tilde{\delta}_n$ for $z,w \in B^k(0;r).$ Let $$\delta_n=\min\{\delta \tilde{C}^{n}r_0,\tilde{\delta}_n\}.$$
As $\|F_n(z)-S_n(z)\|< \delta_n$ for every $z \in \ov{B^k(0;r)}$ and $r_0+\delta_n<r_0+\ep<r<1$ for every $n \ge 0$
\begin{align}\label{inverse 2}
\nonumber \|z-S_n^{-1} \circ F_n(z)\|< &\ep_n \text{ for every } z \in B^k(0;r_0) \\ 
\text{i.e.,} \; \|F_n^{-1}(z)-S_n^{-1}(z)\|< &\ep_n \text{ for every } z \in F_n^{-1}(B^k(0;r_0)).
\end{align}
\textit{Claim: } $\ov{B^k(0;r_0)} \subset \subset F_n^{-1}(B^k(0;r_0))$ for every $n \ge 0.$
 
\medskip\no From (\ref{inverse 1}) and (\ref{inverse 2}), it follows that $S_n^{-1}(B^k(0;r_0)) $ is contained in an $\ep_n-$neighbourhood \\of $F_n^{-1}(B^k(0;r_0))$, i.e., $$B^k(0;r) \subset\subset S_n^{-1}\big(B^k(0;r_0)\big) \subset \Big(F_n^{-1}\big(B^k(0;r_0)\big)\Big)_{\ep_n} .$$ But $$B^k(0;r_0)_{\ep_n} =B^k(0;r_0+\ep_n) \subset B^k(0;r) $$
 Hence for every $n \ge 0$, $$\ov{B^k(0;r_0)} \subset B^k(0;r-\ep_n) \subset \subset F_n^{-1}(B^k(0;r_0)).$$Thus (\ref{inverse 2}) is true for every $z \in \ov{B^k(0;r_0)}$. Further, by the choice of $\delta_n$
\begin{align}\label{inverse 3}
 \|S_n(z)-F_n(z)\|< \delta \tilde{C}^{n}r_0
 \end{align} for every $z \in \ov{B^k(0;r_0)}.$ For $z \in \partial B^k(0;r_0)$ then $\|S_n(z)\|< Cr_0$ hence 
\begin{align}\label{inverse 4}
\|F_n(z)\|< Cr_0+ \delta r_0< \tilde{C}r_0.
\end{align}

\medskip\no 
\textit{Induction hypothesis: } If $z \in \ov{B^k(0;r_0)}$, then $F(n)(z) \in B^k(0;\tilde{C}^{n+1} r_0).$

\medskip\no 
\textit{Initial step: }From (\ref{inverse 4}) note that $F_0(z) \in B^k(0;\tilde{C}r_0)$ for $z \in B^k(0;r_0).$

\medskip\no 
\textit{General step: } Suppose the claim is true for some $n \ge 0$. Let $z \in \partial B^k(0;\tilde{C}^{n+1} r_0).$ From (\ref{inverse 3}) 
\[ \|F_{n+1}(z)\| < C \tilde{C}^{n+1}r_0+\delta \tilde{C}^{n+1}r_0 \le \tilde{C}^{n+2}r_0.\]
Hence $F_{n+1}(B^k(0;\tilde{C}^{n+1}r_0)) \subset B^k(0;\tilde{C}^{n+2}r_0)$, i.e., $F(n+1)(B^k(0;r_0)) \subset B^k(0;\tilde{C}^{n+2}r_0).$

\medskip\no 
Thus $B^k(0;r_0) \subset \om{F}$. Also by similar arguments it follows that for every $z \in \ov{B^k(0;r_0)}$ and $0 \le i \le n$
\begin{align}\label{inverse 5}
F_{n+i} \circ F_{n+i-1}\circ \hdots \circ F_i(z) \in B^k(0;\tilde{C}^{n+1} r_0).
\end{align}
Let $\Omega_n^F=F(n)^{-1}(B^k(0;r_0))$. Now from the above claim, (\ref{inverse 5}) and (\ref{inverse 4}),  $\Omega_n^F \subset \Omega_{n+1}^F$ and 
\[ \om{F} =\bigcup_{n=0}^{\infty} \Omega_n^F.\] So $\om{F}$ is a connected open set containing the origin.

\medskip\no 
Let $\phi_n(z)=S(n)^{-1}F(n)(z) \in \mathsf{ Aut}_0(\mbb C^k).$

\medskip\no 
\textit{Claim: } $\phi_n \to \phi$ on compact subsets of $\om{F}.$

\medskip\no 
Suppose $ K $ is a compact subset of $\om{F}$, it is enough to show that for a given $\eta>0$ there exists $n_0 \ge 0$ such that $$\|\phi_n(z)-\phi_m(z)\|< \eta$$ for every $z \in K$ and $n,m \ge n_0.$

\medskip\no 
Choose $n_0 \ge \max\{n_1,n_2\}$ where $\ep^{n_1}< \eta (1-\ep)$ and $ K \subset \Om_n^{F}$ for every $n \ge n_2.$ Then 
\begin{align*}
\|\phi_n(z)-\phi_m(z)\| \le \sum_{i=n}^{m-1} \|\phi_{i+1}(z)-\phi_i(z)\|\le \sum_{i=n}^{m-1} \|S(i+1)^{-1}F(i+1)(z)-S(i)^{-1}F(i)(z)\|
\end{align*}
for every $z \in K.$ Now $F(i)(z) \in B^k(0;r_0)$ for every $n \le i \le m-1$, i.e., 
\begin{align*}
\|S_{i+1}^{-1}\circ F(i+1)(z)-F(i)(z)\| =& \|(S_{i+1}^{-1}\circ F_{i+1}-\text{Id})(F(i)(z))\| < \ep_{i+1}.
\end{align*}
Thus $S_{i+1}^{-1}\circ F(i+1)(z) \in B^k(0;r)$ for every $z \in K$ and
\[ \|S(i+1)^{-1}F(i+1)(z)-S(i)^{-1}F(i)(z)\|\le M_i \ep_{i+1}< \ep^{i+1}.\]
Hence, for every $z \in K$
\begin{align*}
\|\phi_n(z)-\phi_m(z)\| \le \frac{\ep^{n+1}}{1-\ep}< \eta.
\end{align*}
Since $\phi_n$ converges uniformly on compact subset of $\om{F}$, $\phi$ is holomorphic on $\om{F}.$ 

\medskip\no 
\textit{Claim: } $\phi$ is injective on $\om{F}.$

\medskip\no 
Since $\phi$ is the limit of injective maps, an application of Hurwitz's Theorem shows that either, $\phi$ is injective or $\phi(\mbb C^k)$ has empty interior. Let $\Omega_{n}^S=S(n)^{-1}(B^k(0;r_0))$ then from Proposition \textbf{\ref{uniform bound}}, $\om{S}=\cup_{n=0}^{\infty} \Omega_n^S.$ Further, $\phi_n(\Om_n^F)=\Om_n^S$ for every $n \ge 0.$ By uniform convergence of $\phi_n$'s on relatively compact subsets of $\om{F}$, for a sufficiently small $0<\eta<r_0$ there exists $n$ sufficiently large 
\[ B^k(0;r_0) \subset \Om_n^S \subset \Big(\phi\big(\Om_n^F\big)\Big)_{\eta}.\]
Here $\Big(\phi\big(\Om_n^F\big)\Big)_{\eta}$ is an $\eta-$neighbourhood of $\phi\big(\Om_n^F\big).$ Now if interior of $\phi(\om{F})$ is empty then from above condition $B^k(0;r_0) \subset  B^k(0;\eta)$, which is a contradiction! Hence the claim.

\medskip\no 
\textit{Claim: } $\phi(\om{F})=\om{S}.$ 

\medskip\no 
Suppose $z=\phi(w)$ for some $w \in \om{F}.$ Let $z_n=\phi_n(w)$, i.e., $z_n \in \Om_n^S$ for $n$ sufficiently large. Now $z_n \to z$, i.e., $z \in \om{S}$ or $z \in \partial \om{S}.$ Let $z \in \partial \om{S}$. Since $\phi$ is injective there exists $z_0 \notin \ov{\om{S}}$ such that $z_0 \in \phi(\om{F})$ but arguments similar as above should give $z_0 \in \om{S}$ or $z_0 \in \partial \om{S}.$ This is a contradiction! Thus $\phi(\om{F}) \subset \om{S}.$

\medskip\no 
Suppose $z \in \om{S}$ and $z \notin \phi(\om{F})$ then there exists $\rho>0$ such that $\ov{B^k(z;\rho)} \cap \phi(\om{F})=\emptyset$ and $\ov{B^k(z;\rho)} \subset \om{S}$, i.e., 
$$\ov{B^k(z;\rho)} \subset \Om_n^{S}=\phi_n \big (\Om_n^F \big)$$ for $n \ge n_0.$ For $n$ sufficiently large
$$\phi\big(\Om_n^F\big) \subset \big(\Om_n^S\big)_{\eta} \text{ and }\Om_n^S \subset \Big(\phi\big(\Om_n^F\big)\Big)_{\eta}$$ for $0<\eta<\rho.$ But by choice $z \notin \phi(\om{F})_{\eta}$, i.e., $z \notin \Om_n^S$ for every $n \ge 0$ which is a contradiction! Hence $\phi(\om{F}) = \om{S}.$
\end{proof}
\begin{rem}\label{remark to transcendence}
The choice of $\delta_n$ in the proof of Theorem \textbf{\ref{transcendence}}, depends on the radius of the ball, i.e., $r>0$ where $\seq{S}$ satisfies the \textit{uniform upper--bound} condition. However, the choice of $\delta_n(\tilde{r})$ can be appropriately modified whenever $0< \tilde{r}<r$ to give that $\om{F} \cong \om{S}$ if
\[ \|F_n(z)-S_n(z)\|< \delta_n(\tilde{r})\] for every $z \in B^k(0;\tilde{r}).$ 
\end{rem}
\section{Short \texorpdfstring {$\mathbb{C}^k$}{}s with boundary having upper--box dimension greater than \texorpdfstring {$2k-1$}{}}
 \begin{lem}\label{convergence}
Let $P$ be a hyperbolic polynomial and $J_P(\delta_0)$ denote the $\delta_0-$neighbourhood of the Julia set of $J_P$, then there exists $\{c'_n\}$ a sequence positive real numbers converging to $0$ such that if $|w_n| \le c'_n $ and $z_0 \in \mbb C \setminus P^{-1}(J_P(\delta_0))$ then as $n \to \infty$, either
\[P(z_n)+w_n \to 0 \text{ or } P(z_n)+w_n \to \infty\]
where $z_n=P(z_{n-1})+w_{n-1}$ for $n \ge 1.$
\end{lem}
\begin{proof}
Suppose $z_0$ lies in a compact component of $\mbb C \setminus P^{-1}(J_P(\delta_0))$, say $C$. Let $\delta_1>0$ be chosen such that the $2 \delta_1$ neighbourhood of $P(C)$, i.e., $P(C)_{2\delta_1} \subset C.$ Let $C_1=P(C)_{\delta_1}$, and similarly choose $\delta_2>0$ such that $P(C_1)_{2 \delta_2} \subset C_1.$ Now inductively define $C_n=P(C_{n-1})_{\delta_n}$ for $n \ge 2$ where $\delta_n>0$ is appropriately chosen to satisfy
\[ P(C_n)_{2 \delta_{n+1}} \subset C_n.\] Clearly $\text{diam}(C_n) \to 0$. Hence for $z_0 \in C$ and $|w_n|< \delta_n$, the sequence $z_n \to 0$ as $n \to \infty.$ 

\medskip\no 
A similar argument on the non--compact component of $\mbb C \setminus P^{-1}(J_P(\delta_0))$ gives a sequence $\eta_n$ such that if $|w_n|< \eta_n$, then $z_n \to \infty$ as $n \to \infty.$ Finally choose $c'_n < \min\{ \delta_n, \eta_n\}$ for every $n \ge 1.$ 
\end{proof}
\no Let $p$ be as in Theorem \textbf{\ref{poly_short}} and $P(z)=z^2p(z).$ For a given $\delta>0$ there exists $R>0$ such that $J_P(\delta) \subset D(0;R).$ Consider $\seq{S} \subset \mathsf{ Aut}_0(\mbb C^k)$ as in Remark \textbf{\ref{shift}}. Let $V_R^+$, $V_R^-$ and $V_R$ be defined as:
\begin{align*}
V_R=&\{z \in \mbb C^k: |z_i|\le R \text{ for all } 1 \le i \le k\}\\
V_R^+=&V_R^1 \text{ and } V_R^-=\bigcup_{i=2}^k V_R^i
\end{align*} 
where for a fixed $i$, $1 \le i \le k$ 
\[ V_R^i=\{z \in \mbb C^k: |z_j| \le \max \{|z_i|,R\} \text{ for every } 1 \le j \le k\}.\]
\begin{lem} For $R>0$, sufficiently large
\begin{itemize}
\item[(i)] If $z \in \mbb C^k$, $F(n)(z) \in V_R \cup V_R^+.$

\medskip
\item[(ii)] If $z \in V_R^+$, $S(n)(z) \to \infty.$
\end{itemize}
\end{lem}
\begin{proof}
The arguments are same as in the proof of Lemma \textbf{3.2} and Lemma \textbf{3.3} in \cite{ShortC2}.
\end{proof}  
\no Let
$$N_{C}=\{(z_1,z_2,\hdots,z_k) \in \mbb C^k: z_1 \in \mbb C, |z_i|<C \text{ for all } 2 \le i \le k\}$$ and $U\subset N_{C}$ be defined as:
\begin{align}\label{def of U}
U=\{(z_1,z_2,\hdots,z_k) \in N_{C}: P(z_1) \in J_P(\delta)\}.
\end{align}
Corresponding to the sequence $\seq{S}$, let $K^+_{\{S_n\}}$ and $\J{S}$ denote the following sets:
\begin{align*}
K^+_{\{S_n\}}=\{z \in \mbb C^k: \|S(n)(z)\| \text{ is bounded for every } n \ge 0\} , \; \J{S}=\partial K^+_{\{S_n\}}. 
\end{align*}
\begin{lem}\label{convergence of S_n}
There exists a sequence $\{c_n\}$ of positive real numbers decreasing to zero such that if $|a_{n}|<\min\{|a_{n-1}|^2,c_{n}\}$ for every $n \ge 0$, then $\J{S}\cap N_C \subset U. $
\end{lem}
\begin{proof}
By Lemma \textbf{\ref{convergence}}, there exists a sequence $\{c_n'\}$. Choose $\{c_n\}$ such that $0<c_n'R<c_n$. 

\medskip\no 
If $z$ is in the compact component of $N_C \setminus U.$ Then by the choice of $c_n$'s it follows that $S_1(z)$ is in the compact component of $N_C \setminus U.$ Further repetitive arguments using Lemma \textbf{\ref{convergence}}, shows that $\pi_1 \circ S(n)(z) \to 0.$ Also, $\pi_i \circ S(n+i)=a_{n+i} \pi_1 \circ S(n)(z)$ for $2 \le i \le k.$ Hence it follows that $S(n)(z) \to 0$ as $n \to \infty.$ 

\medskip\no 
If $z$ is in the non--compact component of $N_C \setminus U$, then there are two cases.

\medskip\no 
\textit{Case 1:}If $|\pi_i \circ S(n)(z)| \le R$ for every $2 \le i \le R$ and $n \ge 0$, then the choice of $a_n$'s and Lemma \textbf{\ref{convergence}} assures that $\pi_1\circ S(n)(z) \to \infty$ as $n \to \infty.$ 

\medskip\no 
\textit{Case 2:} Otherwise suppose $|\pi_{i_0} \circ S(\tilde{n})(z)|>R$ for some $\tilde{n} \ge 1$ and $2 \le i_0 \le k$. Also let $|\pi_{i} \circ S(n)(z)|\le R$ for every $0 \le n < \tilde{n}$ and $2 \le i \le k$. If $i_0>2$, then $|\pi_{i_0-1}S(\tilde{n}-1)|>R$ which contradicts the choice of $\tilde{n}$, i.e., $i_0=2.$ Since, $|\pi_i \circ S(\tilde{n}-1)(z)|<R$ for every $2 \le i \le k$ and $|\pi_1 \circ S(\tilde{n}-1)(z)|>R$ it follows that $S(\tilde{n}-1)(z) \in V_R^+$, i.e., $S(n)(z) \to \infty$ as $n \to \infty.$
\end{proof}
\no The proof of Theorem \textbf{5.1} in \cite{PW} relied on the following idea:

\medskip\no 
\begin{adjustwidth*}{.5cm}{.5cm}
`The Fatou--Bieberbach domains $F(n-1)(\Omega^{a_n})$'s constructed for every $n \ge 0$ were converging to $\Omega_{\{F_n\}}$ in the Hausdorff metric on sufficiently large polydiscs in $\mbb C^2.$'
\end{adjustwidth*}

\medskip\no 
However, the proof of Theorem \textbf{\ref{box_k}} does not use this idea. On the contrary, it involves the convergence of forward Julia sets of a sequence of automorphisms to a standard object whose Hausdorff dimension is predetermined. Let us recall a few definitions and standard notations before proceeding to the result:

\medskip\no 
Let $K$ be a compact subset of some metric space, say $X$. For $\ep>0$ let $\cal{B}_{\ep}$ denote the collection of all coverings of $K$ by balls of radius $\ep$, i.e., 
\[ \cal{B}_{\ep}= \big\{ \{B_i\}: K \subset \cup_i B_i \text{ and } B_i=B(p_i; \ep) \text{ for some } p_i \in X  \big\}.\]
For $h \ge 0$ define
\[ \gamma_h^{\ep}(K)=\ep^h \inf_{\cal{B}_{\ep}} \# \{B_i\} \text { and } \mu_h(K)= \limsup_{\ep \to 0} \gamma_k^\ep(K).\]
$\mu_h(K)$ is called the $h-$upper--box content (or the Minkowski content) of $K.$ The upper--box dimension of $K$ is denoted by $\ov{\text{dim}}_B(K)$ and is defined as the unique value of $h \ge 0$ such that 
\[ \mu_{h'}(K)= \begin{cases} 0 \text{ for every } h'>h \text { and }\\ \infty \text{ for every } h< h'. \end{cases}\]
The upper--box dimension of the subset $K$ is always greater than or equal to the Hausdorff dimension (see \cite{Falconer}). 

\medskip\no 
For two compact sets $A,B \subset \mbb C^k$, the definition of Hausdorff distance between $A$ and $B$ is given by 
\[ d_H(A,B)=\max\{d(A,B), d(B,A)\}\]
where 
\[ d(A,B)=\sup_{x \in A} \inf_{y \in B} d(x,y).\] 
\begin{thm}\label{box_k}
There exists an $\alpha_0>1$ and a \short{k}, say $\Omega$ such that the upper--box dimension is greater than or equal to $2(k-1)+\alpha_0$ at every point in the boundary of $\Omega$. Further, $\Omega$ is obtained as a non--autonomous basin of attraction of a sequence of automorphisms in $\mbb C^k.$
\end{thm}
\begin{proof}
\no Note that if $a>0$ and $b>0$ are chosen appropriately the polynomial $p(z)=a z^4+b z^3+ z^2$ satisfies the following properties:
\begin{itemize}
\item[(i)] $p(z)$ is a hyperbolic polynomial with a single attracting cycle only at the origin. This is possible since $z^2$ is hyperbolic with only one component and degree $4$ hyperbolic polynomials form an open subset in the space degree $4$ polynomials 

\medskip
\item[(ii)] By Theorem \textbf{4.4.20} in \cite{Uedabook}, it follows that Fatou set of $ p(z)$ has only two connected components, i.e., the component containing the origin and the component containing infinity. Further, the Julia set of $P$ is the boundary of the Fatou component containing the origin.

\medskip
\item[(iii)] The Hausdorff dimension of $J(p)=\alpha_0>1.$ This follows from Theorem \textbf{{1.4.2}} in \cite{Uedabook}.
\end{itemize}

\medskip\no 
Choose $C>0$ sufficiently small and let $N_C$ (as before) be a $C-$neighbourhood of the $z_1-$axis in $\mbb C^k.$ From the proof of Lemma \textbf{\ref{convergence of S_n}}, for some $\delta>0$ there exists a positive sequence $\{a_n(\delta)\}$ such that for $S_n=S_{a_n(\delta)}$, $$J^+_{\{S_{a_n}\}} \cap {N_C} \subset {J_p(\delta)} \times { D^{k-1}(0;C)}.$$ Let $\cal{J}=J_p \times {D^{k-1}(0;C)}.$ Then the Hausdorff dimension of $\cal{J}$ is equal to $2(k-1)+\alpha_0.$ Let $h_n$ be a sequence increasing to $2(k-1)+\alpha_0.$ The final sequence $S_n$ will be constructed inductively. 

\medskip\no 
\textit{Induction hypothesis:} There exist $(i+1)-$constants $\{a_j \in \mbb R^+: 0 \le j \le i\}$ such that $S_j=S_{a_j}$, $0 \le j \le i$ satisfies the following properties:

\begin{itemize}
\item There exists a finite collection of balls $\cal{B}_i$ of radius $2^{-(i+1)}$ covering $\cal{J}_i=S(i-1)^{-1}({\cal{J}})$ such that every element of $\cal{B}_i$ intersect $\cal{J}_i.$ Further, there exists $\hat{\ep}_i>0$ such that $\gamma_{h_i}^{\hat{\ep}_i}(\cal{J}_i \cap B)>2^{i+1}$ for every $B \in \cal{B}_i.$

\medskip\no 
\item Let $0<\eta_i< \hat{\ep}_i-\ep_i$, where $2\ep_i^{h_i}=\hat{\ep}_i^{h_i}.$ There exists a sequence of positive real numbers $\{a^i_{k}\} $ such that the finite collection $\{S_j: 0 \le j \le i\} $ is completed with $S_{i+k}=S_{a_{i+k}}$ for $k \ge 1$ where $a_{i+k} \le \max\{a_{i+k-1}^3,a^i_{k}\}$ then $$d_H(J^+_{\{S_n\}} \cap S(i-1)^{-1}(N_C), \cal{J}_i)< \eta_i.$$ 
\end{itemize}
\textit{Initial step:} When $i=0$, consider a covering of $\cal{J}$ by balls of radius $1/2$ , say $\cal{B}_0$ such that every element of $\cal{B}_0$ intersect $\cal{J}.$ Further, let $\hat{\ep}_0$ be such that $\gamma_{h_0}^{\hat{\ep}_0}(\cal{J} \cap B)>2$ for every $B \in \cal{B}_0.$ Let $2\ep_0^{h_0}=\hat{\ep}_0^{h_0}$ and $0<\eta_0< \hat{\ep}_0-\ep_0.$ Also consider $U_0=J_p(\eta_0) \times D^{k-1}(0;C).$ Then by Lemma \textbf{\ref{convergence of S_n}}, there exists a sequence of positive real numbers $\{a^0_k\}$ such that if $S_k=S_{a_k}$ where $a_k \le \max \{a_{k-1}^3, a^0_k\}$ for every $k \ge 0$ then $\Omega_{\{S_n\}}$ is a \short{k} and $J^+_{\{S_n\}} \cap N_C \subset U_0.$ Let $S_0=S_{a_0}.$

\medskip\no 
\textit{General step:} Suppose the statement is true for some $i \ge 0.$ Consider a covering of $\cal{J}_{i+1}=S(i)^{-1}(\cal{J})$ by balls of radius $2^{i+2}$ , say $\cal{B}_{i+1}$ such that every element of $\cal{B}_{i+1}$ intersects $\cal{J}_{i+1}.$ Further, $\hat{\ep}_{i+1}$ such that $\gamma_{h_{i+1}}^{\hat{\ep}_{i+1}}(\cal{J} \cap B)>2^{i+2}$ for every $B \in \cal{B}_{i+1}.$ Let $2\ep_{i+1}^{h_{i+1}}=\hat{\ep}_{i+1}^{h_{i+1}}$ and $0<\eta_{i+1}< \hat{\ep}_{i+1}-\ep_{i+1}.$ Further, choose $0<\tilde{\eta}_{i+1}< \eta_0$ such that for $z,w \in U_0$
$$\|S(i)^{-1}(z)-S(i)^{-1}(w)\|< \eta_{i+1} \text { whenever } \|z-w\|< \tilde{\eta}_{i+1}.$$ Let $U_{i+1}=J_p(\tilde\eta_{i+1}) \times D^{k-1}(0;C).$ Then by Lemma \textbf{\ref{convergence of S_n}}, there exists a sequence of positive real numbers $\{a^{i+1}_k\}$ such that if $S_k=S_{a_k}$ where $a_k \le \max \{a_{k-1}^3, a^{i+1}_{k-i}\}$ for every $k \ge i+1$ such that the $\Omega_{\{S_n\}}$ is a \short{k} and $S(i)(J^+_{\{S_n\}}) \cap N_C \subset U_{i+1}$, i.e., 
\[d_H(J^+_{\{S_n\}} \cap S(i)^{-1}(N_C), \cal{J}_{i+1})< \eta_{i+1}.\]
Let $S_{i+1}=S_{a_{i+1}}.$

\medskip\no 
Hence it is possible to obtain $\{S_n\} \subset \mathsf{ Aut}_0(\mbb C^k)$ such that for every $n \ge 0$,
\begin{itemize}
\item There exists $\cal{B}_n$ a finite collection of balls of radius $2^{-(n+1)}$ covering $\cal{J}_n=S(n-1)^{-1}({\cal{J}})$ such that every element of $\cal{B}_n$ intersect $\cal{J}_n.$ Further, there exists $\hat{\ep}_n>0$ such that $\gamma_{h_n}^{\hat{\ep}_n}(\cal{J}_i \cap B)>2^{n+1}$ for every $B \in \cal{B}_n.$

\medskip\no 
\item There exists $0<\eta_n< \hat{\ep}_n-\ep_n$, where $2\ep_n^{h_n}=\hat{\ep}_n^{h_n}$ such that 
\[d_H(J^+_{\{S_n\}} \cap S(n-1)^{-1}(N_C), \cal{J}_{n})< \eta_{n}.\]
\end{itemize} 
Let $z \in J^+_{\{S_n\}}.$ Then for sufficiently large $n \ge n_z \ge 0$, $z \in S(n)^{-1}(N_C).$ Choose $\ep>3.2^{-(n+2)}.$ Let $w \in \cal{J}_{n+1}$ such that $z \in B^k(w; \eta_{n+1}).$ By assumption $\cal{B}_{n+1}$ is a covering by $2^{-(n+2)}$ balls of $\cal{J}_{n+1}.$ Let $B_w$ be the ball in $\cal{B}_{n+1}$ that contains $w$, then $B_w \in B^k(z;\ep).$ Consider any arbitrary covering $\{\tilde{B}_j\}$ of $J^+_{\{S_n\}} \cap B^k(z;\ep)$ by balls of radius $\ep_{n+1}.$ Further, let $\{B'_j\}$ represent the collection of balls with same centers as $\tilde{B}_j$ but radius $\hat{\ep}_{n+1}.$ Since $\eta_{n+1}< \hat{\ep}_{n+1}-\ep_{n+1}$ and 
$$d_H(J^+_{\{S_n\}} \cap S(n)^{-1}(N_C), \cal{J}_{n+1})< \eta_{n+1}$$ $\{B'_j\}$ is a covering of $\cal{J}_{n+1} \cap B_w$. Further, $\hat{\ep}_{n+1}^{h_{n+1}} \# \{B'_j\}> 2^{n+2}$ for every $n \ge n_z.$ Now let $h < 2(k-1)+ \alpha_0.$ Then for sufficiently large $n$, $h_n \ge h$ 
\begin{align}\label{box_main}
 \gamma_{h}^{\ep_{n+1}}(J^+_{\{S_n\}} \cap B^k(z;\ep))>\gamma_{h_{n+1}}^{\ep_{n+1}}(J^+_{\{S_n\}} \cap B^k(z;\ep))>2^{n+1}.
 \end{align}
Since (\ref{box_main}) is true for all $n$, sufficiently large it follows that $\mu_h(J^+_{\{S_n\}} \cap B^k(z;\ep))=\infty$, i.e., the box dimension at $z$ is greater than $h.$ Hence the  upper--box dimension of $J^+_{\{S_n\}}$ at every point is greater than or equal to $2(k-1)+ \alpha_0.$ 
\end{proof}
\begin{rem}
By Theorem \textbf{6.1} in \cite{Wolf}, for $\delta>0$ there exists $a_0(\delta)>0$ such that the forward Julia set ($J_a^+$) of the automorphism $$H_a(z_1,z_2)=(a^2 z_2+p(z_1),z_1)$$ has Hausdorff dimension $h_a \in (2+\alpha_0-\delta,2+\alpha_0+\delta)$ whenever $0<|a|<a_0(\delta)$. Since $$S_a=\cal{L}_a \circ H_a \circ \cal{L}_{a^{-1}}=(aw+p(z),az)$$ where $\cal{L}_a(z_1,z_2)=(z_1,a z_2).$ So the Hausdorff dimension of the forward Julia set of $S_a$ is $h_a.$ Let $\Omega^a$ denote the attracting basin of attraction of $S_a$. From \cite{BS2} it follows that $ J_{a}^+=\partial \Omega^a.$ Theorem \textbf{\ref{box_k}} says that 
\[ \ov{\text{dim}}_H (J_{a_n}^+) \to \ov{\text{dim}}_B (J_{\{F_n\}}^+) .\]
\end{rem}
\begin{prop}\label{constant}
Let $\{S_n\} \subset \mathsf{ Aut}_0(\mbb C^k)$ be the sequence as constructed in the proof of Theorem \textbf{\ref{box_k}}. Then $K^+_{\{S_n\}}$ is connected and
\[ K^+_{\{S_n\}} =\overline{\Omega_{\{S_n\}}} \; \text{ and }\; J^+_{\{S_n\}}=\partial \Omega_{\{S_n\}}.\]
\end{prop}
\begin{proof}
Choose $z_0 \in \J{S}.$ Since $z_0 \in \K{S}$ and $\eta_n \to 0$ as $n \to \infty$, there exists $n_0 \ge 0$, sufficiently large such that 
\[S(n_0)(z_0) \in N_{C-\tilde{\eta}_{n_0}} \text{ and } \eta_{n_0} < \ep/3. \] 
\textit{Claim: } For $z \in J_p \times D^{k-1}(0;C-\tilde{\eta}_{n_0})$ and $r> \tilde{\eta}_{n_0}$ there exists $\theta_1^z$ and $\theta_2^z$ in $B^k(z; r)$ such that $S(n)S(n_0)^{-1}(\theta_1^z) \to 0$ and $S(n)S(n_0)^{-1}(\theta_2^z) \to \infty$ as $n \to \infty.$

\medskip\no Recall that $U_{n_0}=J_p(\tilde{\eta}_{n_0}) \times D^{k-1}(0;C).$ Hence, $$B^k(z;\tilde{\eta}_{n_0}) \subset U_{n_0} \text{ for }z \in J_p \times D^{k-1}(0;C-\tilde{\eta}_{n_0}).$$ Let $z=(z_1,z')$ where $z_1 \in J_p$ and $z' \in D^{k-1}(0;C-\tilde{\eta}_{n_0}).$ Now for $r>r_0>\tilde{\eta}_{n_0}$ consider the points $\theta_t=(z_1+r_0e^{it},z')$ for $t \in [0,2 \pi].$ Then $\theta_t \in B^k(z;r) \setminus B^k(z;\tilde{\eta}_{n_0})$ for every $t.$ Further, there exists $t_1$ and $t_2$ such that $z_1+r_0e^{it_1}$ lies in the compact component of $\mbb C \setminus J_p(\tilde{\eta}_{n_0}) $ and $z_1+r_0e^{it_2}$ lies in the non--compact component respectively. Thus $\theta_1^z=\theta_{t_1}$ lies in the compact component of $N_C \setminus U_{n_0}$ and $\theta_2^z=\theta_{t_2}$ in the non--compact component. By the property of $\seq{S}$'s, it follows that $S(n)S(n_0)^{-1}(\theta_1^z) \to 0$ and $S(n)S(n_0)^{-1}(\theta_2^z) \to \infty$ as $n \to \infty.$

\medskip\no 
Observe that $S(n_0)(z_0) \in N_{C-\tilde{\eta}_{n_0}} \cap U_{n_0}$, i.e., there exists $\tilde{z} \in J_p \times D^{k-1}(0;C-\tilde{\eta}_{n_0}) $ such that $$\|S(n_0)(z_0)-\tilde{z}\|< \tilde{\eta}_{n_0}.$$ Thus $$\|z_0-S(n_0)^{-1}(\tilde{z})\|<{\eta}_{n_0}$$ and $S(n_0)^{-1}(\tilde{z}) \in B^k(z_0;\ep).$ Also by the choice $\eta_{n_0}$, it follows $B^k(S(n_0)^{-1}(\tilde{z}); \eta_{n_0}) \subset B^k(z_0;\ep).$ Now 
\[\ov{B^k(\tilde{z};\tilde{\eta}_{n_0})} \subset S(n_0)\big(B^k(S(n_0)^{-1}(\tilde{z}); \eta_{n_0})\big),\] i.e., there exists $r> \tilde{\eta}_{n_0}$ such that  
\[B^k(\tilde{z};r) \subset S(n_0)\big(B^k(S(n_0)^{-1}(\tilde{z}); \eta_{n_0})\big).\]
Thus from the above claim, there exist $s_1=S(n_0)^{-1}(\theta_1^{\tilde{z}})$ and $s_2=S(n_0)^{-1}(\theta_2^{\tilde{z}}) \in B^k(z;\ep)$ such that $S(n)(s_1) \to 0$ and $S(n)(s_2) \to \infty$ as $n \to \infty.$ Since this is true for any arbitrary $\ep>0$, it follows that $z \in \partial \om{S}.$ Thus the proof.
\end{proof}
\section{Proof of Results \ref{poly_con_1}--\ref{dense}}
\no In this section, we prove some properties of biholomorphic images of non--autonomous basins of attraction at a fixed point that satisfy the \textit{uniform upper--bound} condition. We assume that the non--autonomous basin of attraction is not all of $\mbb C^k$, as in this case it is enough to show existence of Fatou--Bieberbach domains with these properties. Henceforth, we will assume that the non--autonomous basin of attraction is always a proper subset of $\mbb C^k$. Recall the following result from \cite{PW}. We will also have occasions to use the facts stated in the remarks thereafter.
\begin{thm} \label{poly_con}
Let $K_1,K_2, \hdots,K_m$ be pairwise disjoint polynomially convex compact sets in $\mbb C^k$ whose union is polynomially convex, and assume that $K_1,K_2,\hdots,K_l$ are star--shaped $(l \le m).$ Let $\phi_i \in \mathsf{ Aut}(\mbb C^k)$ be automorphisms for $1 \le i \le l$ so that the sets $K_i'=\phi_i(K_i)$ and the sets $K_{l+1},\hdots,K_m$ are pairwise disjoint and their union is polynomially convex. Let $\ep>0.$ Then there exists an automorphism $\phi \in \mathsf{ Aut}(\mbb C^k)$ so that $\|\phi(z)-\phi_i(z)\|< \ep$ for all $z \in K_i$, $1 \le i \le l$ and $\|\phi(z)-z\|< \ep$ for all $z \in K_j$, $l+1 \le j \le m.$
\end{thm}

\begin{rem}\label{properties_poly_con}\leavevmode\vspace{-.1\baselineskip}
\begin{enumerate}
\item[(i)] The union of a polynomially convex compact set and a finite set of points is polynomially convex.

\medskip
\item[(ii)] If $K_1 \cup K_2$ is polynomially convex and compact, $K_1 \cap K_2=\emptyset$, and $K_1' \subset K_1$ is polynomially convex and compact then $K_1' \cup K_2$ is polynomially convex.

\medskip
\item[(iii)] A polynomially convex compact set has a neighbourhood basis consisting of polynomially convex compact sets.

\item[(iv)]  The union of two disjoint polynomially convex compact set, that can be separated by two disjoint convex compact sets is polynomially convex.
\end{enumerate}
\end{rem}
\begin{proof}[Proof of Theorem \ref{poly_con_1}]
Since $\seq{S} $ satisfies the condition of Theorem \textbf{\ref{transcendence}}, the sequence $\seq{\delta}$ as in the proof of Theorem \textbf{\ref{transcendence}} gives a convergent series. So let $\ep_n =\sum_{i=0}^n \delta_i$ for every $n \ge 0$ and $\ep=\sum_{i=0}^{\infty} \delta_i.$ Moreover, there exists $0<r_0<1$ such that $$\om{S}=\cup_{i=0}^{\infty}\Om_i^S \text{ where }\Om_i^S=S(i)^{-1}(B^k(0;r_0)).$$ Without loss of generality assume that  there exists $p \in \mbb C^k$ and $R>0$ such that $\ep-$neighbourhood of $K$, i.e., $K_{\ep} \subset B^k(p;R)$ and $\ov{B^k(p;R)} \cap \ov{B^k(0;r_0+\ep)}=\emptyset.$ Let $\bar{B}=\ov{B^k(0;r_0)}.$ Also let $p_0=0.$

\medskip\no 
\textit{Induction hypothesis: }For every $i \ge 0$ there exist $i-$many automorphisms in $\mathsf{ Aut}_0(\mbb C^k)$ such that the following are true:
\begin{align*}
&{\norm{F_j-S_j}}_{\bar{B}} < \delta_j , \\
F(j)(p_j) \subset \bar{B} &\text{ and } F(j)(K) \subset K_{\ep_j} \subset\ov{B(p;R)} \subset \mbb C^k\setminus\bar{B} 
\end{align*} 
for every $0 \le j \le i.$

\medskip\no 
\textit{Initial step: } By Remark \textbf{\ref{properties_poly_con}}(iv) $\bar{B} \cup K$ is polynomially convex. Since, $S_0(\bar{B}) \subset \bar{B}$,  $S_0(\bar{B}) \cup K$ is also polynomially convex. Hence, by Theorem \textbf{\ref{poly_con}}, for $\delta_0$ there exists $\phi \in \mathsf{ Aut}_0(\mbb C^k)$ such that 
\[ {\norm{\phi-S_0}}_{\bar{B}} < \delta_0 \text{ and } {\norm{\phi-\text{Id}}}_K< \delta_0.\]
Let $F_0=\phi$. Note that $\phi(K) \subset K_{\delta_0}$ and $\phi(p_0)\in \bar{B}.$

\medskip\no 
\textit{General step: } Let $\mu_{i+1}=\delta_{i+1}/2.$ By the same reasoning as before $\bar{B} \cup F(i)(K)$ is polynomially convex and $S_{i+1}(\bar{B}) \cup F(i)(K)$ is polynomially convex. Hence, by Theorem \textbf{\ref{poly_con}} there exists $\phi \in \mathsf{ Aut}_0(\mbb C^k)$ such that 
\[ {\norm{\phi-S_{i+1}}}_{\bar{B}} < \mu_{i+1} \text{ and } {\norm{\phi-\text{Id}}}_{F(i)(K)}< \mu_{i+1}.\]
From (\ref{inverse 1a}) in the proof of Theorem \textbf{\ref{transcendence}}, 
$\phi(\bar{B}) \subset \bar{B}.$
Hence, $\phi\circ F(i)(p_j) \subset \bar{B}$ for every $0 \le j \le i.$ Now if $F(i)(p_{i+1}) \in \phi^{-1}(\bar{B})$, then consider $F_{i+1}=\phi.$ 

\medskip\no 
Otherwise, if $F(i)(p_{i+1}) \notin \phi^{-1}(\bar{B})$, i.e., $F(i)(p_{i+1}) \notin \bar{B}.$ From Remark \textbf{\ref{properties_poly_con}}(i), $\bar{B} \cup F(i)(K) \cup F(i)(p_{i+1})$ is polynomially convex.  Let $\tau_{i+1} \in \phi^{-1}(\bar{B})\setminus (F(i)(K)\cup \bar{B})$, then $\bar{B} \cup F(i)(K) \cup \tau_{i+1}$ is also polynomially convex. There exists $1>\rho>0$ such that for $z,w \in (\bar{B} \cup F(i)(K))_1$, i.e., on a radius $1-$ neighbourhood of $\bar{B} \cup F(i)(K)$, $$\norm{\phi(z)-\phi(w)}< \mu_{i+1}\text{ whenever }\norm{z-w}< \rho.$$ Hence, by Theorem \textbf{\ref{poly_con}} there exists $\psi \in \mathsf{ Aut}_{0}(\mbb C^k)$ such that 
$${\norm{\psi-\text{Id}}}_{\bar{B} \cup F(i)(K)}< \rho \text{ and } \psi(F(i)(p_{i+1})) =\tau_{i+1}\in \phi^{-1}(\bar{B}).$$
Consider $F_{i+1}=\phi \circ \psi.$ From the construction $F(i+1)(p_{i+1}) \in \bar{B}.$ 
For $z \in \bar{B}$, then $\norm{\psi(z)-z}< \rho$ and $\psi(z) \in \ov{B(0;1+r_0)}.$ Thus by continuity of $\phi$
\[ \|\phi\circ \psi(z)-\phi(z)\|< \mu_{i+1} \text { and } \|\phi(z)-S_{i+1}(z)\|< \mu_{i+1},\] i.e.,
\[ \|F_{i+1}(z)-S_{i+1}(z)\|< \delta_{i+1}.\]
Similar arguments for $z \in F(i)(K)$ gives
\[ \|F_{i+1}(z)-z\|< \delta_{i+1},\] i.e., $F(i+1)(K) \subset (K_{\ep_i})_{\delta_{i+1}}=K_{\ep_{i+1}}.$
Hence the induction statement is true for $i+1.$

\medskip\no 
Now by Theorem\textbf{\ref{transcendence}}, $\om{F}$ is biholomorphic to $\om{S}$. Also $\{p_j\} \subset \om{F}$ and $K \cap \om{F}=\emptyset.$
\end{proof}
\begin{proof}[Proof of Corollary \ref{compact}]
Without loss of generality consider $p=0$ and $K$ sufficiently away from the origin. Let $\{p_j\}$ be a dense sequence in $\mbb C^k \setminus K$. Then by Theorem \textbf{\ref{poly_con_1}}, there exists a sequence of automorphisms  $\seq{F} \in \mathsf{ Aut_0}(\mbb C^k)$ such that $\om{F}$ is biholomorphic to $\om{S}$ and $\{p_j\} \subset \om{F}$ and $\om{F} \cap K=\emptyset.$ But $\om{F}$ is open and hence the proof.
\end{proof}
\begin{cor}
Given a sequence of automorphisms  $\seq{S} \in \mathsf{ Aut_0}(\mbb C^k)$ that satisfy the uniform upper--bound condition at the origin, there exists a biholomorphism of $\om{S}$ (say $\Phi$), such that the $2k-$dimensional Hausdorff measure of $\partial \Phi(\om{S})$ is non--zero.
\end{cor}
\begin{proof}
Let $D=\bar{(D)}^{\mathrm{o}} \subset \subset \mbb C$ be a simply connected domain in $\mbb C$ such that $\partial D$ has non--zero two dimensional Hausdorff measure. Then $K=\ov{D^k}=\ov{D\times \cdots \times D} \subset \mbb C^k$ is a polynomially convex compact set with non--zero $2k-$dimensional Hausdorff measure. By Corollary \textbf{\ref{compact}}, the result follows. 
\end{proof}
\begin{proof}[Proof of Corollary \ref{Runge}]
From Corollary \textbf{\ref{compact}}, for any given sequence $\seq{S}$ there exists $\Phi_1(\om{S}) \subset \mbb C^* \times \mbb C^{k-1}.$ From Theorem \textbf{\ref{Runge_pre}}, there exists $\Phi_2 \in \mathsf{ Aut}(\mbb C^* \times \mbb C^{k-1})$ such that $Y \subset \Phi_2^{-1} \circ \Phi_1(\om{S}). $ Let $\Phi=\Phi_2^{-1} \circ \Phi_1.$ Then $\Phi(\om{S})$ is not Runge.
\end{proof}
%
\begin{proof}[Proof of Theorem \ref{dense}]
Choose $\ep_n \to 0$ as $n \to \infty.$ 

\medskip\no \textit{Case 1:} When $m=\infty.$ 

\medskip\no \textit{Induction hypothesis: } For every $i \ge 0$, there exist
\begin{itemize}
\item $(i+1)-$automorphisms $\{F_j \in \mathsf{ Aut}(\mbb C^k): 0 \le j \le i\}$,
\item Two set of distinct points $P^i=\{p_j \in \mbb C^k: 0 \le j \le i\}$ and $Q^i=\{q_j \in \mbb C^k: 0 \le j \le i\}$,
\item A set of positive numbers $\Gamma^i=\{\rho_j \in \mbb R^+: 0 \le j \le i\}$,
\end{itemize}   
with the following properties:
\begin{itemize}
\item[(i)] $\ov{B^k(q_j; \rho_j)} \cap \ov{B^k(q_k; \rho_k)}=\emptyset$ for $0 \le j \neq k \le i.$ 
\item[(ii)] $F(i)(p_j)=q_j$ for all $0 \le j \le i.$
\item[(iii)] $F_j(q_k)=q_k$ for all $0 \le k \le j$ and $0 \le j \le i$
\item[(iv)]  For every $z \in \ov{B^k(q_k ; \rho_k)}$ $$\| F_j(z)-S_j(z-q_k)-q_k\|<\delta_j(\rho_k)$$ whenever $0 \le k \le j$ and $0 \le j \le i.$ Here $\delta_j(\rho_k)$ is as observed in Remark \textbf{\ref{remark to transcendence}}.
\item [(v)] $B^i=\cup_{j=0}^i \ov{B^k(q_j ; \rho_j)}$ is polynomially convex.
\item[(vi)] For $t \in B^k(0;i) \setminus F(i)^{-1}(B^i)$ and for every $j$, $0 \le j \le i$
\[ \text{dist}\Big(t, F(i)^{-1}\big(\ov{B^k(q_j; \rho_j)}\big)\Big)< \ep_i.\]
\item[(vii)] $P^j \subset P^{j+1}$, $Q^j \subset Q^{j+1}$ and $\Gamma^j \subset \Gamma^{j+1}$ where $0 \le j \le i-1.$
\end{itemize}

\medskip\no 
\textit{Initial step: }Let $p_0=q_0=0$ be the origin, $F_0=S_0$ and $\rho_0=r$, as in Theorem \textbf{\ref{transcendence}}. Since $i=0$, all the conditions are true.

\medskip\no \textit{General step: }Suppose all the assumptions are true for some $i \ge 0.$ Let $K_i=F(i)^{-1}(B^i)$, i.e., $K_i$ is polynomially convex. For every $0 \le j \le i+1$ consider a set of points $$T^j=\{t^j_{l}: 1 \le l \le m_j\} \subset B^k(0;i+1)\setminus int(K_i)$$ for some $m_j \ge 1$, such that $T^j \cap T^k=\emptyset $ whenever $0 \le j \neq k \le i+1.$ Also for each $j$, $0 \le j \le i+1$  
$$\text{dist}(t, T^j)< \ep_{i+1} \text{ whenever }t \in B^k(0;i+1)\setminus int(K_i).$$ Let $\mathscr{T}^{i+1}=\cup_{j=0}^{i+1} T^j$ and $\uptau^j=F(i)(T^j).$
Choose a point $p_{i+1} \in \mbb C^k \setminus (\mathscr{T}^{i+1} \cup K_i) $ and let $q_{i+1}=F(i)(p_{i+1}).$ Now there exists $\rho_{i+1}>0$ such that
\[ \ov{B^k(q_{i+1}; \rho_{i+1})} \cap \Big({\textstyle \bigcup\limits_{j=0}^{\infty}} \uptau^j \cap B^i\Big)=\emptyset.\] Further, from Remark \textbf{\ref{properties_poly_con}}(iii) we have the following:
\begin{itemize}
\item By appropriately modifying $\rho_{i+1}$ we have that $\ov{B^k(q_{i+1}; \rho_{i+1})} \cup B^i$ is polynomially convex.

\medskip\no 
\item There exists $\mu>0$, $\bigcup_{j=0}^{i+1}\ov{B^k(q_j;\rho_j+\mu)}$ is polynomially convex and $\ov{B^k(q_j;\rho_j+\mu)} \cap \ov{B^k(q_k;\rho_k+\mu)}=\emptyset$ whenever $0 \le j\neq k \le i+1.$

\medskip\no 
\item Let $\psi_j(z)=S_{i+1}(z-q_j)+q_j.$ Then $\bigcup_{j=0}^{i+1}\psi_j\big(\ov{B^k(q_j;\rho_j+\mu)}\big)$ is polynomially convex
\end{itemize}  
 By Theorem \textbf{\ref{poly_con}} there exists $\phi \in \mathsf{ Aut}(\mbb C^k)$ such that for every $0 \le j \le i+1$
 \[\|\phi(z)-\psi_j(z)\|< \mu_{i+1}\]
where $\mu_{i+1}=\min\{\mu,\delta_{i+1}(\rho_j)/2: 0 \le j \le i+1\}$ and $\phi(q_j)=q_j.$ By continuity of $\phi$, there exists $\tilde{\mu}_{i+1}< \mu_{i+1}$ such that on $\bigcup_{j=0}^{i+1}\psi_j\big(\ov{B^k(q_j;\rho_j+\mu)}\big)$
\begin{align}\label{continuity}
\|\phi(z)-\phi(w)\|< \mu_{i+1} \text{ whenever } \|z-w\|< \tilde\mu_{i+1}.
\end{align}
 Again, by Theorem \textbf{\ref{poly_con}} there exists $\psi \in \mathsf{ Aut}(\mbb C^k)$ such that 
\[ \|\psi(z)-z\|< \tilde\mu_{i+1}\] on each $\ov{B^k(q_j,\rho_j)}$ and $\psi(\uptau^j) \subset \phi^{-1}({B^k(q_j,\rho_j)})$ for every $0 \le j \le i+1.$ Further, $\psi(q_j)=q_j.$ Let $F_{i+1}=\phi \circ \psi.$

\medskip\no 
Clearly, the collection $\{F_j: 0 \le j \le i+1\}$ satisfies all the properties (i)--(iii), (v) and (vii). Let $z \in \ov{({B^k(q_j,\rho_j)})}$, then $\psi(z) \in \ov{B^k(q_j,\rho_j+\mu)}.$ From (\ref{continuity})
\[ \|F_{i+1}(z)-\phi(z)\|< \mu_{i+1}, \text{ i.e., } \|F_{i+1}(z)-\psi_j(z)\|<\delta_{i+1}(\rho_j)\] for every $0 \le j \le i+1.$ Hence property (iv) is true.

\medskip\no Also, $F(i+1)(T^j) \subset B^k(q_j,\rho_j)$ for every $0 \le j \le i+1$ and by choice of $T_j$'s property (vi) is also satisfied.

\medskip\no Let $\seq{S^i}$ denote the sequence $S^i_n=S_{i+n}$ for every $i \ge 0.$ Now from the sequence $\seq{F}$ obtained the non--autonomous basin of attraction at every point $q_i$, i.e., $\om{F^i}\cong \om{S^i}$ for $i \ge 0.$ Since $\om{S^i}=S(i)(\om{S})$, it follows that ${\om{F^i}} \cong \om{S}.$ Now by construction ${\om{F^i}} \cap {\om{F^j}}=\emptyset$ for $i \neq j.$ Also for any given $\ep>0$, there exists $n_0 \ge 0$ such that $\ep_{n_0}< \ep$, hence for every $i \ge 0$ and $t \notin \mbb C^k \setminus {\cup_{i=0}^{\infty}} {\om{F^i}}$
\[ \text{dist} (t,\partial {\om{F^i}})< \ep.\] Thus $t \in \partial  {\om{F^i}}$ for every $i \ge 0.$

\medskip\no 
\textit{Case 2:} When $m < \infty.$

\medskip\no For $p_{m+i}=q_m$ for every $i \ge 1$ and follow the same procedure as for the infinite case.
\end{proof}
\section{Proof of Theorem \ref{hausdorff}}
\no In this section we use Theorem \textbf{\ref{transcendence}} to prove that there exists  biholomorphic images of non--autonomous basins of attraction at a point satisfying the \textit{uniform upper--bound }condition with completely chaotic boundary. The technique is adapted from Theorem \textbf{1.1} from \cite{PW}.
\begin{proof}[Proof of Theorem \ref{hausdorff}]
Let $D=int(\bar{D})$ be a simply connected domain in $\mbb C$ such that the Hausdorff dimension of $\partial D$ is $2.$ Let $K=D^k=D \times D \cdots \times D$, then the Hausdorff dimension of $\partial K$ is $2k.$ Also for any $p \in \mbb C^K$ and $\ep>0$ there exists an appropriate affine transformation $\phi_{p,\ep}$ such that $p \in \phi_{p,\ep}(K) \subset B^k(p;\ep).$ Let $K(p;\ep)=\phi_{p,\ep}(K).$ Let $r>0$ and $\seq{\delta}$ be as obtained in Theorem \textbf{\ref{transcendence}}. Further, let
\[ \tilde{\delta}_n=\sum_{j=n}^{\infty} \delta_j.\] Choose $\ep_n \to 0$ as $n \to \infty$.

\medskip\no 
\textit{Induction hypothesis: }For every $i \ge 0$, there exist
\begin{itemize}
\item $(i+1)-$automorphisms $\{F_j \in \mathsf{ Aut}_0(\mbb C^k): 0 \le j \le i\}$,
\item Three set of distinct points $P^i=\{p_j^i \in \mbb C^k: 0 \le j \le n(i)\}$, $Q^i=\{q_j: 0 \le j \le i\}$ and $T^i=\{t_j^i \in \mbb C^k: 0 \le j \le m(i)\}$, where $m(i),n(i)>0$ for every $i \ge 0$
\item Two set of positive numbers $\Gamma^i=\{\rho_j \in \mbb R^+: 0 \le j \le i\}$ and $R^i=\{R_j \in \mbb R^+: 0 \le j \le i\}$

\end{itemize}   
with the following properties:
\begin{itemize}
\item[(i)] $\|F_i-S_i\|< \delta_i$ on $\bar{B}$.
\item[(ii)] $B^k(0;i+1)\setminus F(i)(\bar{B}) \neq \phi.$
\item[(iii)] $F(i)(T^i) \in \bar{B}.$
\item[(iv)] $K^i=F(i-1)^{-1}\big(\cup_{j=0}^{n(i)} K(p_j^i ; \rho_j) \big) \cup K^{i-1}$ is polynomially convex. 
\item [(v)] For every $p \in B^k(0;i+1)\setminus int(K^i)$, $\text{dist}(p,T^i)< \ep_i.$
\item[(vi)] For every $p \in B^k(0;i+1) \setminus F(i)^{-1}(\bar{B})$, $\text{dist}(p,K^i)< \ep_i.$
\item[(vii)] $\bar{B} \cap \ov{B^k(q_i;R_i)}=\emptyset$, $\text{dist}(\bar{B},\ov{B^k(q_i;R_i)})> \tilde{\delta}_i$ and $B^k(q_j;R_j) \subset B^k(q_{j+1}; R_{j+1})$ for every $0 \le j \le i-1.$ 
\item[(viii)]$R^j \subset R^{j+1}$, $Q^j \subset Q^{j+1}$ and $\Gamma^j \subset \Gamma^{j+1}$ where $0 \le j \le i-1.$
\item[(ix)] $F(i)(K^i) \subset B^k(q_i;R_i).$
\end{itemize}

\medskip\no 
\textit{Initial step: } Let $P^0=\{p_j^0 \in {B^k(0;1)} \setminus \bar{B}: 1 \le j \le n(0)\}$ for some $n(0) \ge 1$ such that for any point in $$p \in \ov{B^k(0;1)} \setminus B,\;\text{dist}(p,P^0)< \ep_0.$$ Further, from Remark \textbf{\ref{properties_poly_con}} there exists $\rho_0>0$ such that the following are true:
\begin{enumerate}
\item[(a)] $\bar{B} \cup \{\ov{B^k(p_j^0, \rho_0)}: 1 \le j \le n(0)\}$ is polynomially convex.
\item[(b)] $\ov{B^k(p_j^0, \rho_0)} \cap \ov{B^k(p_l^0, \rho_0)} \neq \emptyset$ for $1 \le j \neq l \le n(0)$ and
\[ B^0=\bigcup_{j=1}^{n(0)} \ov{B^k(p_j^0, \rho_0)} \text{ is polynomially convex.}\]
\end{enumerate}
Let $$K^0= \bigcup_{j=1}^{n(0)} K(p_j^0, \rho_0),$$ then again from Remark \textbf{\ref{properties_poly_con}}, it follows that $K^0$ and $\bar{B} \cup K^0$ is polynomially convex. Let $T^0=\{t_j^0 \in B^k(0;1) \setminus K^0: 1\le j \le m(0)\}$ for some $m(0)\ge 1$ be a collection of points such that for every $$p \in \ov{B^k(0;1)} \setminus int(K^0), \;\;\text{dist}(p,T^0)< \ep_0.$$ Choose $q_0 \in \mbb C$ and $R_0>0$, sufficiently large such that $\text{dist}(\bar{B},\ov{B^k(q_0;R_0)})> \tilde{\delta}_0.$ Since, $\bar{B} \cup \ov{B^k(q_0;R_0)}$ is polynomially convex and $S_0(\bar{B}) \subset B$, from Remark \textbf{\ref{properties_poly_con}}(ii) it follows that $S_0(\bar{B}) \cup \ov{B^k(q_0;R_0)}$ are polynomially convex. Further, let $0<\mu_0<\delta_0/2$ be chosen appropriately such that
\[\|S_0(z)-S_0(w)\|< \delta_0/2 \text{ whenever }\|z-w\|< \mu_0\] for every $z \in \ov{B^k(0;1+\delta_0)}.$ Hence from Theorem \textbf{\ref{poly_con}}, there exists $\phi_1, \phi_2, \phi_3 \in \mathsf{ Aut}_0(\mbb C^k)$ such that  
\begin{align}\label{hausdorff-1}
\|\phi_1-S_0\|_{\bar{B}_{\delta_0}}< \delta_0/2  \text{ and } \|\phi_1-\text{Id}\|_{\ov{B^k(q_0;R_0)}}<\delta_0/2,
\end{align}
\begin{align}\label{hausdorff-2}
\|\phi_2-\text{Id}\|_{\bar{B}}< \mu_0 /2 \text{ and } \phi_2(B^0) \in B^k(q_0;R_0-\delta_0)\; \text{and }
\end{align}  
\begin{align}\label{hausdorff-3}
\|\phi_3-\text{Id}\|_{\phi_2(\bar{B}) \cup \phi_2(K^0)}< \mu_0 /2 \text{ and } \phi_3 \circ \phi_2(T^0) \in \phi_1^{-1}(B).
\end{align} 
\textit{Claim: } $F_0=\phi_1 \circ \phi_3 \circ \phi_2$ satisfies the induction hypothesis for $i=0.$

\medskip\no Note that by choice $F_0$ satisfies properties (ii)--(v) and (vii)--(ix). Let $z \in \bar{B}$
\[ \|\phi_3\circ \phi_2(z)-z\| \le \|\phi_3 \circ \phi_2(z)-\phi_2(z)\|+\|\phi_2(z)-z\| < \mu_0.\] Also $\phi_3 \circ \phi_2(\bar{B}) \subset B_{\delta_0}$ and from the choice of $\mu_0$ it follows that 
\[ \|S_0 \circ \phi_3\circ \phi_2(z)-S_0(z)\| \le \delta_0/2.\] Thus
\[ \|\phi_1\circ \phi_3\circ \phi_2(z)-S_0\circ \phi_3\circ \phi_2(z)\|< \delta_0/2.\] Hence $\|F_0-S_0\|_{\bar{B}}< \delta_0$, i.e., (i) is true. Also from relation (\ref{inverse 1a}) in the proof of Theorem \textbf{\ref{transcendence}}, $ \bar{B} \subset F_0^{-1}(\bar{B}) $, i.e., (vi) is also true.

\medskip\no 
\textit{Induction step: }Suppose the conditions are true for some $i \ge 0.$  Let 
$$\tilde{P}^{i+1}=\{\tilde{p}_j^{i+1} \in {B^k(0;{i+2})} \setminus \big(K^i \cup F(i)^{-1}(\bar{B})\big): 1 \le j \le n(i+1)\}$$ 
for some $n(i+1) \ge 1$ such that for any point in $$p \in \ov{B^k(0;i+2)} \setminus \big(int(K^i) \cup F(i)^{-1}(B)\big),\;\text{dist}(p,\tilde{P}^{i+1})< \ep_{i+1}.$$ Let $P^{i+1}=F(i)(\tilde{P}^{i+1})$ and $p^{i+1}_j=F(i)(\tilde{p}^{i+1}_j)$ for $1 \le j \le n(i+1).$ Further, from Remark \textbf{\ref{properties_poly_con}} there exists $\rho_{i+1}>0$ such that the following are true:
\begin{enumerate}
\item[(a)] $\bar{B} \cup F(i)(K^i) \cup \{\ov{B^k(p_j^{i+1}, \rho_{i+1})}: 1 \le j \le n(i+1)\}$ is polynomially convex.
\item[(b)] $\ov{B^k(p_j^{i+1}, \rho_{i+1})} \cap \ov{B^k(p_l^{i+1}, \rho_{i+1})} \neq \emptyset$ for $1 \le j \neq l \le n(i+1)$ and
\[ B^{i+1}= F(i)(K^i)\cup \Big( \bigcup_{j=1}^{n(i+1)} \ov{B^k(p_j^{i+1}, \rho_{i+1})}\Big) \text{ is polynomially convex.}\]
\end{enumerate}
Let $$K^{i+1}= K^i\cup F(i)^{-1} \Big(\bigcup_{j=1}^{n(i+1)} K(p_j^{i+1}, \rho_{i+1})\Big),$$ then again from Remark \textbf{\ref{properties_poly_con}}, it follows that $K^{i+1}$ and $F(i)^{-1}(\bar{B}) \cup K^{i+1}$ is polynomially convex. Let $T^{i+1}=\{t_j^{i+1} \in B^k(0;i+2) \setminus K^{i+1}: 1\le j \le m(i+1)\}$ for some $m(i+1)\ge 1$ be a collection of points such that for every $$p \in \ov{B^k(0;i+2)} \setminus int(K^{i+1}), \;\;\text{dist}(p,T^{i+1})< \ep_{i+1}.$$ Choose $q_{i+1} \in \mbb C$ and $R_{i+1}>0$, sufficiently large such that $\text{dist}(\bar{B},\ov{B^k(q_{i+1};R_{i+1})})> \tilde{\delta}_{i+1}$ and $B(q_i;R_i) \subset B(q_{i+1};R_{i+1}-\delta_{i+1}).$ Since, $\bar{B} \cup \ov{B^k(q_{i+1};R_{i+1})}$ is polynomially convex and $S_0(\bar{B}) \subset B$, from Remark \textbf{\ref{properties_poly_con}}(ii) it follows that $S_{i+1}(\bar{B}) \cup \ov{B^k(q_{i+1};R_{i+1})}$ is polynomially convex. Let $0<\mu_{i+1}<\delta_{i+1}/2$ be chosen appropriately such that
\[\|S_i(z)-S_i(w)\|< \delta_{i+1}/2 \text{ whenever }\|z-w\|< \mu_{i+1}\] for every $z \in \ov{B^k(0;i+2+\delta_{i+1})}.$ Hence from Theorem \textbf{\ref{poly_con}}, there exists $\phi_1, \phi_2, \phi_3 \in \mathsf{ Aut}_0(\mbb C^k)$ such that 
\begin{align}\label{hausdorff-general-1}
\|\phi_1-S_{i+1}\|_{\bar{B}_{\delta_{i+1}}}< \delta_{i+1}/2  \text{ and } \|\phi_1-\text{Id}\|_{\ov{B^k(q_{i+1};R_{i+1})}}<\delta_{i+1}/2,
\end{align}
\begin{align}\label{hausdorff-general-2}
\|\phi_2-\text{Id}\|_{\bar{B}}< \mu_{i+1} /2 \text{ and } \phi_2(B^{i+1}) \in B^k(q_{i+1};R_{i+1}-\delta_{i+1}+\mu_{i+1}/2)\; \text{and }
\end{align}  
\begin{align}\label{hausdorff-general-3}
\|\phi_3-\text{Id}\|_{\phi_2(\bar{B}) \cup \phi_2 \circ F(i)(K^{i+1})}< \mu_{i+1} /2 \text{ and } \phi_3 \circ \phi_2(T^{i+1}) \in \phi_1^{-1}(B).
\end{align} 
\textit{Claim: } $F_{i+1}=\phi_1 \circ \phi_3 \circ \phi_2$ satisfies the induction hypothesis for $i+1.$

\medskip\no Note that by choice $F_{i+1}$ satisfy properties (ii)--(v), (vii) and (viii). Let $z \in \bar{B}$
\[ \|\phi_3\circ \phi_2(z)-z\| \le \|\phi_3 \circ \phi_2(z)-\phi_2(z)\|+\|\phi_2(z)-z\| < \mu_{i+1}.\] Also $\phi_3 \circ \phi_2(\bar{B}) \subset B_{\delta_{i+1}}$ and from the choice of $\mu_{i+1}$ it follows that 
\[ \|S_{i+1} \circ \phi_3\circ \phi_2(z)-S_{i+1}(z)\| \le \delta_{i+1}/2.\] Thus
\[ \|\phi_1\circ \phi_3\circ \phi_2(z)-S_{i+1}\circ \phi_3\circ \phi_2(z)\|< \delta_{i+1}/2.\] Hence $\|F_{i+1}-S_{i+1}\|_{\bar{B}}< \delta_{i+1}$, i.e., (i) is true. Also from relation (\ref{inverse 1a}), $ \bar{B} \subset F_{i+1}^{-1}(\bar{B}) $, i.e., $F(i)^{-1}(\bar{B}) \subset F(i+1)^{-1}(\bar{B}).$ Thus for any $z \in B^k(0;i+1)\setminus F(i+1)^{-1}(\bar{B})$ means either $z \in int(K^i) \subset K^{i+1}$ or $\text{dist}(z,F(i)^{-1}({P}^{i+1}))< \ep_{i+1}.$ But $F(i)^{-1}({P}^{i+1}) \in K^{i+1}$, hence (vi) is satisfied. Finally as $F(i)(K^{i+1}) \subset B^{i+1}$, it follows from (\ref{hausdorff-general-1})--(\ref{hausdorff-general-3}), $F_{i+1}(F(i)(K^{i+1})) \subset B^k(q_{i+1},R_{i+1})$, which proves (ix).

\medskip\no 
Hence we obtain a sequence $\seq{F} \subset \mathsf{ Aut}_0(\mbb C^k)$ such that:
\begin{itemize}
\item From property (i) and Theorem \textbf{\ref{transcendence}}, $\om{F} \cong \om{S},$
\item From property (iii) and (v), $K \subset \partial \om{F}$ where $K=\cup_{i=0}^{\infty} \partial K^i,$
\item From property (vi), $K$ is a dense subset of $\partial \om{F}.$
\end{itemize} 
Now by construction, the $2k-$dimensional measure is non--zero at every point of $K$ and hence on $\partial \om{F}.$ Thus the proof.
\end{proof}
\bibliographystyle{amsplain}
	\bibliography{ref}
	\end{document}